\documentclass[english,12pt]{amsart}

\usepackage{amssymb}

\usepackage[utf8]{inputenc}
\usepackage[T1]{fontenc}
\usepackage[english]{babel}

\usepackage[inline]{enumitem}

\usepackage{euscript}

\usepackage{fouriernc}
\usepackage[scaled=0.875]{helvet}
\usepackage{courier}

\usepackage{MnSymbol}
\usepackage{ marvosym }

\usepackage[babel=true]{csquotes}

\usepackage[marginparwidth=2cm]{geometry}
\usepackage[usenames,dvipsnames,svgnames,table]{xcolor}
\usepackage[colorlinks=true,urlcolor=DarkBlue,linkcolor=DarkRed,citecolor=DarkGreen,pagebackref=true]{hyperref}

\usepackage[normalem]{ulem}
\usepackage[marginpar]{todo}

\usepackage{tikz}
\usepackage{fp}
\usetikzlibrary{calc}
\usetikzlibrary{shapes,arrows}
\usetikzlibrary{fit,positioning}
\usetikzlibrary{patterns,decorations.pathreplacing}
\makeatletter
\def\calcLength(#1,#2)#3{%
\pgfpointdiff{\pgfpointanchor{#1}{center}}%
             {\pgfpointanchor{#2}{center}}%
\pgf@xa=\pgf@x%
\pgf@ya=\pgf@y%
\FPeval\@temp@a{\pgfmath@tonumber{\pgf@xa}}%
\FPeval\@temp@b{\pgfmath@tonumber{\pgf@ya}}%
\FPeval\@temp@sum{(\@temp@a*\@temp@a+\@temp@b*\@temp@b)}%
\FProot{\FPMathLen}{\@temp@sum}{2}%
\FPround\FPMathLen\FPMathLen5\relax
\global\expandafter\edef\csname #3\endcsname{\FPMathLen}
}
\makeatother

\usepackage{xkeyval}
\usepackage{moreverb}
\usepackage{epic}
\usepgfmodule{shapes,plot,decorations}
\usepackage{accents}

\usepackage{epsfig}
\usepackage{pinlabel}
\usepackage{subfig}

\newcommand{\Z}{\mathbb{Z}}

\newcommand{\R}{\mathbb{R}}
\newcommand{\C}{\mathcal{C}} 

\newcommand{\G}{\Gamma}
\newcommand{\g}{\gamma}

\renewcommand{\C}{\mathcal{C}}
\newcommand{\Uc}{\mathcal{U}}

\renewcommand{\O}{\Omega}

\renewcommand{\S}{\mathbb{S}}

\newcommand{\ie}{i.e.\ }

\definecolor{light-gray}{gray}{0.80}
\definecolor{dark-gray}{gray}{0.30}
\newcommand{\Lshaded}{light-gray}
\newcommand{\Dshaded}{dark-gray}

\renewcommand{\geq}{\geqslant}

\newcommand{\light}{-}

\usepackage{aurical}

\theoremstyle{plain}
\newtheorem{theorem}{Theorem}[section]

\newtheorem{proposition}[theorem]{Proposition}
\newtheorem{corollary}[theorem]{Corollary}
\newtheorem{lemma}[theorem]{Lemma}

\newtheorem{theointro}{Theorem}[section]
\newtheorem{quesintro}{Question}[section]

\theoremstyle{definition}
\newtheorem{example}[theorem]{Example}
\newtheorem{definition}[theorem]{Definition}

\theoremstyle{remark}
\newtheorem{remark}[theorem]{Remark}

\newtheorem{remarkintro}{Remark}


\title[AdS strictly GHC-regular groups which are not lattices]{Anti-de Sitter strictly GHC-regular groups \\ which are not lattices}

\author{
Gye-Seon Lee
}
\address{Mathematisches Institut, Ruprecht-Karls-Universit\"{a}t Heidelberg, Germany}
\email{lee@mathi.uni-heidelberg.de}

\author{
Ludovic Marquis
}
\address{Univ Rennes, CNRS, IRMAR - UMR 6625, F-35000 Rennes, France}
\email{ludovic.marquis@univ-rennes1.fr}

\begin{document}


\let\oldmarginpar\marginpar
\renewcommand\marginpar[1]{\-\oldmarginpar[\raggedleft\tiny #1]%
{\raggedright\tiny #1}}

\newcommand\crule[3][black]{\textcolor{#1}{\rule{#2}{#3}}}
\newcommand\mybox[1][black]{\textcolor{#1}{\rule{2cm}{0.5cm}}}


\begin{abstract}
For $d=4, 5, 6, 7, 8$, we exhibit examples of $\mathrm{AdS}^{d,1}$ strictly GHC-regular groups which are not quasi-isometric to the hyperbolic space $\mathbb{H}^d$, nor to any symmetric space. This provides a negative answer to Question 5.2 in \cite{survey_all} and disproves Conjecture 8.11 of Barbot--Mérigot \cite{barbot_merigot}.

 We construct those examples using the Tits representation of well-chosen Coxeter groups. On the way, we give an alternative proof of Moussong's hyperbolicity criterion \cite{moussong} for Coxeter groups built on Danciger--Guéritaud--Kassel \cite{DGK} and find examples of Coxeter groups $W$ such that the space of strictly GHC-regular representations of $W$ into $\mathrm{PO}_{d,2}(\mathbb{R})$ up to conjugation is disconnected.
\end{abstract}


\maketitle

\tableofcontents

\section{Introduction}

Let $\R^{d,e}$ be the vector space $\R^{d+e}$ endowed with a non-degenerate symmetric bilinear form $\langle \cdot, \cdot \rangle_{d,e}$ of signature $(d,e)$, and let $q$ be the associated quadratic form. A coordinate representation of $q$ with respect to some basis of $\R^{d,e}$ is:
$$
q(x) = x_1^2 + \cdots + x_d^2 - x_{d+1}^2 - \cdots - x_{d+e}^2
$$
The $(d+e-1)$-dimensional \emph{pseudohyperbolic space $\mathbb{H}^{d,e-1}$} is the quotient of the hyperquadric  
$$
\mathcal{Q} = \{ x \in \R^{d,e} \mid q(x) = -1 \}
$$
by the involution $x \mapsto -x$ and is endowed with the semi-Riemannian metric induced by the quadratic form $q$.  It is a complete semi-Riemannian manifold of constant sectional curvature $-1$. The isometry group of $\mathbb{H}^{d,e-1}$ is $\mathrm{PO}_{d,e}(\R)$, and so $\mathbb{H}^{d,e-1}$ may be identified with $\mathrm{PO}_{d,e}(\R)/\mathrm{O}_{d,e-1}(\R)$. It will be more convenient for us to work with the projective model of $\mathbb{H}^{d,e-1}$, namely:
$$
\mathbb{H}^{d,e-1} = \{ [x] \in \mathbb{P}(\mathbb{R}^{d,e})  \mid q(x) < 0 \}
$$ 
In particular, for $e=2$, the pseudohyperbolic space $\mathbb{H}^{d,e-1}$ is called the $(d+1)$-dimensional \emph{anti-de Sitter space $\mathrm{AdS}^{d,1}$}.

\medskip

Let $\Gamma$ be a finitely presented group. A representation $\rho:\G \to \mathrm{PO}_{d,2}(\R)$ is \emph{strictly GHC-regular}
 if it is discrete, faithful and preserves an \emph{acausal} (i.e. \emph{negative}) subset\footnote{We denote by $\partial \mathbb{H}^{d,e-1}$ the boundary of $\mathbb{H}^{d,e-1}$ in $\mathbb{P}(\mathbb{R}^{d,e})$. In the case $e=2$, the boundary $\partial \mathbb{H}^{d,e-1} = \partial \mathrm{AdS}^{d,1}$ is called the \emph{Einstein universe}, denoted $\mathrm{Ein}^d$. A subset $\Lambda$ of $\partial \mathbb{H}^{d,e-1}$ is \emph{negative} if it lifts to a cone of $\mathbb{R}^{d,e} \setminus \{0\}$ on which all inner products $\langle \cdot, \cdot \rangle_{d,e}$ of noncollinear points are negative. By a \emph{cone} we mean a subset of $\mathbb{R}^{d,e}$ which is invariant under multiplication by positive scalars.} $\Lambda_\rho$ in the boundary $\mathrm{Ein}^{d}$ of $\mathrm{AdS}^{d,1}$ such that $\Lambda_\rho$ is homeomorphic to a $(d-1)$-dimensional sphere and the action of $\G$ on the convex hull\footnote{An element of $\mathrm{PGL}(\mathbb{R}^{d,e})$ is \emph{proximal} if it admits a unique attracting fixed point in $\mathbb{P}(\mathbb{R}^{d,e})$. The (proximal) \emph{limit set} of a discrete subgroup $\Gamma_0$ of $\mathrm{PGL}(\mathbb{R}^{d,e})$ is the closure of the set of attracting fixed points of elements of $\Gamma_0$ which are proximal. In fact, the $\rho(\Gamma)$-invariant subset $\Lambda_\rho$ must be the limit set of $\rho(\G)$, and the convex hull of $\Lambda_\rho$ is well-defined because $\Lambda_\rho$ is acausal.} of $\Lambda_\rho$ in $\mathrm{AdS}^{d,1}$ is cocompact. An $\mathrm{AdS}^{d,1}$ \emph{strictly GHC-regular group} (simply \emph{strictly GHC-regular group}) is a subgroup of $\mathrm{PO}_{d,2}(\R)$ which is the image of a strictly GHC-regular representation.

\medskip

The first examples of strictly GHC-regular representations are \emph{Fuchsian representations}. Namely, consider any uniform lattice $\G$ of $\mathrm{PO}_{d,1}(\R)$, which is isomorphic to $\mathrm{O}_{d,1}^+(\R)$, and let $\rho_0 : \Gamma \to \mathrm{PO}_{d,2}(\R)$ be the restriction to $\Gamma$ of the natural inclusion $\mathrm{PO}_{d,1}(\R) \simeq \mathrm{O}_{d,1}^+(\R) \hookrightarrow \mathrm{PO}_{d,2}(\R)$ (hence $\mathrm{O}_{d,1}(\R)$ corresponds to the stabilizer of a point $p_0 \in \mathrm{AdS}^{d,1}$ in $\mathrm{PO}_{d,2}(\R)$). The representation $\rho_0$ is strictly GHC-regular and the convex hull of $\Lambda_{\rho_0}$ is a copy of the hyperbolic $d$-space $\mathbb{H}^d$ inside $\mathrm{AdS}^{d,1}$. Moreover, any small deformation $\rho_t : \Gamma \to \mathrm{PO}_{d,2}(\R)$ of the representation $\rho_0$ is strictly GHC-regular (see Barbot--Mérigot \cite{barbot_merigot}). Even more, Barbot \cite{barbot_close} proved that every representation  $\rho : \G \to  \mathrm{PO}_{d,2}(\R)$ in the connected component of $\mathrm{Hom} ( \Gamma, \mathrm{PO}_{d,2}(\R))$ containing $\rho_0$ is strictly GHC-regular.

\medskip

We will not use the following point of view, but we remark that strictly GHC-regular representations of torsion-free groups are exactly the holonomies of GHMC $\mathrm{AdS}$ manifolds (see Mess \cite{mess} and Barbot \cite{barbot_quasiFusch=GHMC}). We quickly recall this terminology: an $\mathrm{AdS}$ manifold $M$ is \emph{\underline{g}lobally \underline{h}yperbolic spatially \underline{c}ompact} (GHC) if it contains a compact \emph{Cauchy hypersurface}, i.e. a space-like hypersurface $S$ such that all inextendible time-like lines intersect $S$ exactly once. It is \emph{globally hyperbolic maximal compact} (GHMC) if in addition every isometric embedding of $M$ into another GHC $\mathrm{AdS}$ manifold of the same dimension is an isometry.

\medskip

Recently, Barbot--Mérigot \cite{barbot_merigot} found an alternative description of $\mathrm{AdS}^{d,1}$ strictly GHC-regular groups. In the case when $\G$ is isomorphic to the fundamental group of a closed, negatively curved Riemannian $d$-manifold, a representation $\rho:\G \to \mathrm{PO}_{d,2}(\R)$ is strictly GHC-regular if and only if it is $P_1^{d,2}$-Anosov, where $P_1^{d,2}$ is the stabilizer of an isotropic line $l_0 \in \R^{d,2}$ in $\mathrm{PO}_{d,2}(\mathbb{R})$ (see Labourie \cite{labourie_anosov} or Guichard--Wienhard \cite{MR2981818} for more background on Anosov representations) and the limit set of $\rho(\G)$ is a topological $(d-1)$-sphere in the boundary $\mathrm{Ein}^{d}$ of $\mathrm{AdS}^{d,1}$.

\medskip

Very recently, Danciger--Guéritaud--Kassel \cite{DGK} introduced a notion of \emph{convex cocompactness in}  $\mathbb{H}^{d,e-1}$. A discrete subgroup $\G$ of $\mathrm{PO}_{d,e}(\R)$ is \emph{$\mathbb{H}^{d,e-1}$-convex cocompact} if it acts properly discontinuously and cocompactly on some properly convex\footnote{A subset $\mathcal{C}$ of $\mathbb{P}(\mathbb{R}^{d,e})$ is \emph{properly convex} if its closure $\overline{\mathcal{C}}$ is contained and convex in some affine chart.} closed subset $\mathcal{C}$ of $\mathbb{H}^{d,e-1}$ with non-empty interior\footnote{In the case $\Gamma$ is irreducible, any non-empty $\Gamma$-invariant convex subset of $\mathbb{P}(\mathbb{R}^{d,e})$ has non-empty interior, and so “$\mathcal{C}$ with non-empty interior” may be replaced by “$\mathcal{C}$ non-empty”.} so that $\overline{\mathcal{C}} \setminus \mathcal{C}$ does not contain any non-trivial projective segment. A main result of \cite{DGK} is as follows: an irreducible discrete subgroup $\Gamma$ of $\mathrm{PO}_{d,e}(\R)$ is $\mathbb{H}^{d,e-1}$-convex cocompact if and only if $\Gamma$ is Gromov-hyperbolic, the natural inclusion $\Gamma  \hookrightarrow  \mathrm{PO}_{d,e}(\R)$ is $P_1^{d,e}$-Anosov, and the limit set $\Lambda_\Gamma$ is negative. 

\medskip

In particular, for any discrete faithful representation $\rho:\G \to \mathrm{PO}_{d,2}(\R)$, we have that $\rho$ is strictly GHC-regular if and only if $\rho(\Gamma)$ is $\mathrm{AdS}^{d,1}$-convex cocompact (which implies that $\G$ is Gromov-hyperbolic) and the Gromov boundary of $\G$ is homeomorphic to a $(d-1)$-dimensional sphere.

\medskip

In \cite{survey_all}, the authors asked the following:

\begin{quesintro}[Question 5.2 of \cite{survey_all}]\label{que:main}
Assume that $\rho: \G \to \mathrm{PO}_{d,2}(\R)$ is a strictly GHC-regular representation. Is $\G$ isomorphic to a uniform lattice of $\mathrm{PO}_{d,1}(\R)$?
\end{quesintro}

The positive answer to this question was conjectured by Barbot and Mérigot (see Conjecture 8.11 of \cite{barbot_merigot}). In this paper, we give a negative answer to Question \ref{que:main}.

\begin{theointro}\label{thm:main}
For $d=4,5,6,7,8$, there exists an $\mathrm{AdS}^{d,1}$ strictly GHC-regular group which is not isomorphic (even not quasi-isometric) to a uniform lattice in a semi-simple Lie group.
\end{theointro}

The counterexamples are the Tits representations of well-chosen Coxeter groups. 

\begin{remarkintro}
For $d=2$, there exist no examples as in Theorem \ref{thm:main} by work of Tukia \cite{MR961162}, Gabai \cite{MR1189862} and Casson--Jungreis \cite{MR1296353}: if $\G$ is a Gromov-hyperbolic group whose boundary is homeomorphic to $\mathbb{S}^{1}$, then $\G$ admits a geometric action on $\mathbb{H}^2$, \ie there exists a properly discontinuous, cocompact and isometric action of $\G$ on $\mathbb{H}^{2}$ (for a proof see Chapter 23 of Dru\c{t}u--Kapovich \cite{book_kapo_drutu}). In the case $d=3$, recall that Cannon’s conjecture claims that if $\G$ is a Gromov-hyperbolic group whose boundary is homeomorphic to $\mathbb{S}^2$, then $\G$ admits a geometric action on $\mathbb{H}^3$. If Cannon's conjecture is true, then any $\mathrm{AdS}^{3,1}$ strictly GHC-regular group is isomorphic to a uniform lattice of $\mathrm{PO}_{3,1}(\mathbb{R})$. For Coxeter groups, Cannon's conjecture is in fact known to be true (see Bourdon--Kleiner \cite{MR3019076}).
\end{remarkintro}

An important invariant of a Coxeter group $W$ is \emph{the signature $s_W$ of $W$}. It means the signature of the Cosine matrix of $W$ (see Section \ref{section:Coxetergroups} for the basic terminology of Coxeter groups). Recall that the signature of a symmetric matrix $A$ is the triple $(p,q,r)$ of the positive, negative and zero indices of inertia of $A$.

\medskip

The outline of the proof of Theorem \ref{thm:main} is as follows. First, observe that if the signature of a Coxeter group $W$ is $(d,2,0)$, then the Tits representation $\rho$ of $W$ is a discrete faithful representation of $W$ into $\mathrm{PO}_{d,2}(\R)$.  Second, we prove Theorem \ref{thm:dgk} implying that if in addition $W$ is Gromov-hyperbolic and the Coxeter diagram of $W$ has no edge of label $\infty$, then the representation $\rho : W \to \mathrm{PO}_{d,2}(\R)$ is $\mathrm{AdS}^{d,1}$-convex cocompact. Remark that Theorem \ref{thm:dgk} is a generalization of Theorem 8.2 of \cite{DGK}. Third, in order to prove that $\rho$ is strictly GHC-regular, we only need to check that the Gromov boundary of $W$ is a $(d-1)$-dimensional sphere. Coxeter groups satisfying all these properties exist:

\begin{theointro}\label{thm:specifications}
Every Coxeter group $W$ in Tables \ref{examples_dim4}, \ref{examples_dim5}, \ref{examples_dim6}, \ref{examples_dim7}, \ref{examples_dim8} satisfies the following: 

\begin{itemize}
\item the group $W$ is Gromov-hyperbolic;
\item the Gromov boundary $\partial W$ of $W$ is homeomorphic to a $(d-1)$-sphere;
\item the signature $s_W$ of $W$ is $(d,2,0)$.
\end{itemize}
Here, $d$ denotes the rank of $W$ minus $2$.
\end{theointro}

Theorem \ref{thm:main} is then a consequence of the following:

\begin{theointro}\label{thm:tits_rep}
Let $W$ be a Coxeter group in Tables \ref{examples_dim4}, \ref{examples_dim5}, \ref{examples_dim6}, \ref{examples_dim7}, \ref{examples_dim8}. Then the Tits representation $\rho: W \to \mathrm{PO}_{d,2}(\R)$ is strictly GHC-regular and the group $W$ is not quasi-isometric to a uniform lattice of a semi-simple Lie group.
\end{theointro}

\begin{remarkintro}
With Selberg's lemma, Theorem \ref{thm:tits_rep} also allows us to obtain a finite index torsion-free subgroup $\G$ of $W$ so that $\mathrm{AdS}^{d,1}/\rho(\G)$ is a GHMC AdS manifold and $\G$ is not quasi-isometric to a uniform lattice of a semi-simple Lie group.
\end{remarkintro}

Applying the same method as for Theorem \ref{thm:specifications} and Theorem \ref{thm:tits_rep}, we can prove:

\begin{theointro}\label{thm:tits_rep_H}
Let $W$ be a Coxeter group in Table \ref{Hexamples_dim4}. Then the following hold:
\begin{itemize}
\item the group $W$ is Gromov-hyperbolic;
\item the Gromov boundary of $W$ is homeomorphic to a $3$-dimensional sphere;
\item the signature of $W$ is $(5,1,0)$;
\item the Tits representation $\rho: W \to \mathrm{PO}_{5,1}(\R)$ is quasi-Fuchsian;
\item the group $W$ is not quasi-isometric to a uniform lattice of a semi-simple Lie group.
\end{itemize}
\end{theointro}

Here, a representation $\rho : \Gamma \to \mathrm{PO}_{d+1,1}(\mathbb{R})$ is called \emph{$\mathbb{H}^{d+1}$-quasi-Fuchsian} (simply \emph{quasi-Fuchsian}) if it is discrete and faithful, $\rho(\Gamma)$ is $\mathbb{H}^{d+1}$-convex cocompact and the Gromov boundary of $\Gamma$ is homeomorphic to a $(d-1)$-dimensional sphere. 

\begin{remarkintro}
In \cite{MR1396674}, Esselmann proved that each Coxeter group $W$ in Table \ref{examples_Esselmann} admits a geometric action on the hyperbolic space $\mathbb{H}^4$ such that a fundamental domain for the $W$-action on $\mathbb{H}^4$ is a convex polytope whose combinatorial type is the product of two triangles. Esselmann’s examples in Table \ref{examples_Esselmann} are adapted to obtain the Coxeter groups in Tables \ref{Hexamples_dim4}, \ref{examples_dim4}, \ref{examples_dim5}, \ref{examples_dim6}, \ref{examples_dim7}, \ref{examples_dim8}. In particular, the nerve of each Coxeter group in Tables \ref{examples_Esselmann}, \ref{Hexamples_dim4}, \ref{examples_dim4} is isomorphic to the boundary complex of the dual polytope of the product of two triangles. We are currently classifying Coxeter groups with nerve isomorphic to the boundary complex of the dual polytope of the product of two simplices of dimension $\geqslant 2$, and in the meanwhile, we found the Coxeter groups in Tables \ref{Hexamples_dim4}, \ref{examples_dim4}, \ref{examples_dim5}, \ref{examples_dim6}, \ref{examples_dim7}, \ref{examples_dim8}.
\end{remarkintro}

We denote by $\chi(\G, \mathrm{PO}_{d,2}(\mathbb{R}) )$ the $\mathrm{PO}_{d,2}(\mathbb{R})$-character variety of $\G$ and by $\chi^{\mathrm{s}}(\G, \mathrm{PO}_{d,2}(\mathbb{R}) )$ the space of strictly GHC-regular characters in $\chi(\G, \mathrm{PO}_{d,2}(\mathbb{R}) )$. In \cite{barbot_close}, Barbot asked the following:  

\begin{quesintro}[Discussion after Theorem 1.5 of \cite{barbot_close}]\label{que:disconnected}
Assume that $\G$ is a lattice of $\mathrm{PO}_{d,1}(\R)$. Is the space $\chi^{\mathrm{s}}(\G, \mathrm{PO}_{d,2}(\mathbb{R}) )$ connected? In other words, does every strictly GHC-regular character deform to a Fuchsian character?
\end{quesintro}

We can extend Question \ref{que:disconnected} to the case when $\Gamma$ is Gromov-hyperbolic:
\begin{quesintro}\label{que:disconnected_ex}
Let $\G$ be a Gromov-hyperbolic group. Assume that $\chi^{\mathrm{s}}(\G, \mathrm{PO}_{d,2}(\mathbb{R}) )$ is non-empty. Is the space $\chi^{\mathrm{s}}(\G, \mathrm{PO}_{d,2}(\mathbb{R}) )$ connected?
\end{quesintro}

The following theorem provides a negative answer to Question \ref{que:disconnected_ex}.

\begin{theointro}\label{thm:disconnected}
Let $W$ be a Coxeter group in Tables \ref{table:barbot2_not-poincare_hyp_dim4} or \ref{table:barbot2_not-poincare_hyp_dim6}, and let $d$ be the rank of $W$ minus 3. Then there exist two characters $[\rho_1], [\rho_2] \in \chi^{\mathrm{s}}(W, \mathrm{PO}_{d,2}(\mathbb{R}))$ such that there is \emph{no} continuous path from $[\rho_1]$ to $[\rho_2]$ in $\chi(W, \mathrm{PO}_{d,2}(\mathbb{R}))$, let alone in the subspace $\chi^{\mathrm{s}}(W, \mathrm{PO}_{d,2}(\mathbb{R}))$.
\end{theointro}

Similarly, we denote by $\chi(\G, \mathrm{PO}_{d+1,1}(\mathbb{R}) )$ the $\mathrm{PO}_{d+1,1}(\mathbb{R})$-character variety of $\G$ and by $\chi^{\mathrm{q}}(\G, \mathrm{PO}_{d+1,1}(\mathbb{R}))$ the space of quasi-Fuchsian characters in $\chi(\G, \mathrm{PO}_{d+1,1}(\mathbb{R}) )$. It follows from work of Barbot--Mérigot \cite{barbot_merigot} and Barbot \cite{barbot_close} that $\chi^{\mathrm{s}}(\G, \mathrm{PO}_{d,2}(\mathbb{R}))$ is a union of connected components of $\chi(\G, \mathrm{PO}_{d,2}(\mathbb{R}))$. However, $\chi^{\mathrm{q}}(\G, \mathrm{PO}_{d+1,1}(\mathbb{R}) )$ is open but not closed in $\chi(\G, \mathrm{PO}_{d+1,1}(\mathbb{R}))$ in general. Nevertheless, we are able to prove:

\begin{theointro}\label{thm:disconnected_H}
Let $W$ be a Coxeter group in Tables \ref{table:barbot2_Quasi_Fuchsian_dim4} or \ref{table:barbot2_Quasi_Fuchsian_dim6}, and let $d$ be the rank of $W$ minus 3. Then there exist two characters $[\rho_1], [\rho_2] \in \chi^{\mathrm{q}}(W, \mathrm{PO}_{d+1,1}(\mathbb{R}) )$ such that there is \emph{no} continuous path from $[\rho_1]$ to $[\rho_2]$ in $\chi(W, \mathrm{PO}_{d+1,1}(\mathbb{R}))$.
\end{theointro}

However, Question \ref{que:disconnected} is still an open problem.

\begin{remarkintro}
Each Coxeter group $W$ in Tables \ref{table:barbot2_not-poincare_hyp_dim4}, \ref{table:barbot2_not-poincare_hyp_dim6}, \ref{table:barbot2_Quasi_Fuchsian_dim4}, \ref{table:barbot2_Quasi_Fuchsian_dim6} is not quasi-isometric to $\mathbb{H}^d$ even though it is a Gromov-hyperbolic group whose boundary is homeomorphic to $\S^{d-1}$. 
\end{remarkintro}

\begin{remarkintro}
In \cite{tumarkin_n+3}, Tumarkin proved that each Coxeter group $W$ in Table \ref{table:Tumarkin_dim4} (resp. Table \ref{table:Tumarkin_dim6}) admits a geometric action on the hyperbolic space $\mathbb{H}^4$ (resp. $\mathbb{H}^6$). Tumarkin’s examples in Tables \ref{table:Tumarkin_dim4}, \ref{table:Tumarkin_dim6} are adapted to obtain the Coxeter groups in Tables \ref{table:barbot2_not-poincare_hyp_dim4}, \ref{table:barbot2_not-poincare_hyp_dim6}, \ref{table:barbot2_Quasi_Fuchsian_dim4}, \ref{table:barbot2_Quasi_Fuchsian_dim6}. In particular, the nerves of Coxeter groups in Table \ref{table:Tumarkin_dim4} (resp. Table \ref{table:Tumarkin_dim6}) are isomorphic to the nerves of Coxeter groups in Tables \ref{table:barbot2_not-poincare_hyp_dim4}, \ref{table:barbot2_Quasi_Fuchsian_dim4} (resp. Tables \ref{table:barbot2_not-poincare_hyp_dim6}, \ref{table:barbot2_Quasi_Fuchsian_dim6}).
\end{remarkintro}

In Section \ref{section:Moussong}, we also recover Moussong's hyperbolicity criterion for Coxeter groups (see Theorem \ref{thm:moussong}) built on \cite{DGK}. 

\subsection*{Acknowledgments}

We are thankful for helpful conversations with Ryan Greene, Olivier Guichard, Fanny Kassel and Anna Wienhard. We thank Daniel Monclair for introducing Question 5.2 in \cite{survey_all}  to the second author a long time ago. Finally, we would like to thank the referee for carefully reading this paper and suggesting several improvements.

G.-S. Lee was supported by the European Research Council under ERC-Consolidator Grant 614733 and by DFG grant LE 3901/1-1 within the Priority Programme SPP 2026 “Geometry at Infinity”, and he acknowledges support from U.S. National Science Foundation grants DMS 1107452, 1107263, 1107367 “RNMS: Geometric structures And Representation varieties” (the GEAR Network). 

\section{Coxeter groups and Davis complexes}\label{section:Coxetergroups}

A \emph{pre-Coxeter system} is a pair $(W,S)$ of a group $W$ and a set $S$ of elements of $W$ of order $2$ which generates $W$. A \emph{Coxeter matrix} $M=(M_{st})_{s,t \in S}$ on a set $S$ is a symmetric matrix with the entries $M_{st} \in \{1,2, \dotsc, m, \dotsc,\infty \}$ such that the diagonal entries $M_{ss}=1$ and others $M_{st} \neq 1$. For any pre-Coxeter system $(W,S)$, we have a Coxeter matrix $M_W$ on $S$ such that each entry $(M_W)_{st}$ of $M_W$ is the order of $st$. 
To any Coxeter matrix $M=(M_{st})_{s,t \in S}$ is associated a presentation for a group $\hat{W}_M$: the set of generators for $\hat{W}_M$ is $S$ and the set of relations is
$  \{ (st)^{M_{st}}=1  \mid  M_{st} \neq \infty \}$. A pre-Coxeter system $(W,S)$ is called a \emph{Coxeter system} if the surjective homomorphism $\hat{W}_{M_W} \to W$, defined by $s \mapsto s$, is an isomorphism. If this is the case, then $W$ is a \emph{Coxeter group} and $S$ is a \emph{fundamental set of generators}. We denote the Coxeter group $W$ also by $W_S$ or $W_{S,M}$ to indicate the fundamental set of generators or the Coxeter matrix of $(W,S)$. The \emph{rank} of $W_S$ is the cardinality $\sharp S$ of $S$. 

\medskip

The \emph{Coxeter diagram} of $W_{S,M}$ is a labeled graph $\mathsf{G}_W$ such that (\emph{i}) the set of nodes (i.e. vertices) of $\mathsf{G}_W$ is the set $S$; (\emph{ii}) two nodes $s,t \in S$ are connected by an edge $\overline{st}$ of $\mathsf{G}_W$ if and only if $M_{st} \in \{3,\dotsc, m, \dotsc,\infty \}$; (\emph{iii}) the label of the edge $\overline{st}$ is $M_{st}$. It is customary to omit the label of the edge $\overline{st}$ if $M_{st} = 3$. A Coxeter group $W$ is \emph{irreducible} if the Coxeter diagram $\mathsf{G}_W$ is connected. The \emph{Cosine matrix} of $W_{S,M}$ is an $S \times S$ symmetric matrix $\mathsf{C}_W$ whose entries are:
 $$(\mathsf{C}_W)_{st} = -2\cos \left( \frac{\pi}{M_{st}} \right) \quad \,\,\textrm{for every }  s,t \in S$$ 

\medskip

An irreducible Coxeter group $W_S$ is \emph{spherical} (resp. \emph{affine}, resp. \emph{Lann{\'e}r}) if for every $s \in S$, the $(s,s)$-minor of $\mathsf{C}_W$ is positive definite and the determinant $\det(\mathsf{C}_W)$ of $\mathsf{C}_W$ is positive (resp. zero, resp. negative). Remark that every irreducible Coxeter group $W$ is spherical, affine or \emph{large}, \ie there is a surjective homomorphism of a finite index subgroup of $W$ onto a free group of rank $\geqslant 2$ (see Margulis--Vinberg \cite{margu_vin}). 
For a Coxeter group $W$ (not necessarily irreducible), each connected component of the Coxeter diagram $\mathsf{G}_W$ corresponds to a Coxeter group, called a \emph{component} of the Coxeter group $W$. A Coxeter group $W$ is \emph{spherical} (resp. \emph{affine}) if each component of $W$ is spherical (resp. affine).
We will often refer to Appendix \ref{classi_diagram} for the list of all the irreducible spherical, irreducible affine and Lann{\'e}r Coxeter diagrams. 

\medskip

Let $(W,S)$ be a Coxeter system. For each $T\subset S$, the subgroup $W_T$ of $W$ generated by $T$ is called a \emph{special subgroup of $W$}. It is well-known that $(W_T,T)$ is a Coxeter system. A subset $T \subset S$ is said to be “\emph{something}” if the Coxeter group $W_T$ is “something”. For example, the word “something” can be replaced by “spherical”, “affine”, “Lannér” and so on. Two subsets $T, U \subset S$ are \emph{orthogonal} if $M_{tu}=2$ for every $t \in T$ and every $u \in U$. This relationship is denoted $T \perp U$.

\medskip

A \emph{poset} is a partially ordered set. The \emph{nerve} $N_W$ of $W$ is the poset of all non-empty spherical subsets of $S$ partially ordered by inclusion. We remark that the nerve $N_W$ is an abstract simplicial complex. Recall that an \emph{abstract simplicial complex} $\mathcal{S}$ is a pair $(V,\mathcal{E})$ of a set $V$, which we call the \emph{vertex set} of $\mathcal{S}$, and a collection $\mathcal{E}$ of (non-empty) finite subsets of $V$ such that (\emph{i}) for each $v \in V$, $\{ v \} \in \mathcal{E}$; (\emph{ii}) if $T \in \mathcal{E}$ and if $\varnothing \neq U \subset T$, then $U \in \mathcal{E}$. An element of $\mathcal{E}$ is a \emph{simplex} of $\mathcal{S}$, and the \emph{dimension} of a simplex $T$ is $\sharp T - 1$. The vertex set of the nerve $N_{W}$ of a Coxeter group $W_S$ is the set of generators $S$ and a non-empty subset $T \subset S$ is a simplex of $N_{W}$ if and only if $T$ is spherical.

\medskip

The \emph{opposite} poset to a poset $\mathcal{P}$ is the poset $\mathcal{P}^{op}$ with the same underlying set and with the reversed order relation. Given a convex polytope $P$, let $\mathcal{F}(\partial P)$ denote the set of all non-empty faces of the boundary $\partial P$ of $P$ partially ordered by inclusion.  A polytope $P$ is \emph{simplicial} (resp. \emph{simple}) if $\mathcal{F}(\partial P)$ (resp. $\mathcal{F}( \partial P)^{op}$) is a simplicial complex.

\begin{remark}[The nerve of a geometric reflection group, following Example 7.1.4 of \cite{davis_book}]
Let $P$ be a simple convex polytope in $\mathbb{X}= \R^d$ or $\mathbb{H}^d$ with dihedral angles integral submultiples of $\pi$. To the polytope $P$ is associated a Coxeter matrix $M = (M_{st})_{s, t \in S}$ on a set $S$: (\emph{i}) the set $S$ consists of all facets\footnote{A \emph{facet} of a convex polytope $P$ is a face of codimension $1$.} of $P$; (\emph{ii}) for each pair of distinct facets $s, t$ of $P$, if $s, t$ are adjacent\footnote{Two facets $s, t$ of $P$ are \emph{adjacent} if the intersection of $s$ and $t$ is a face of codimension $2$.} and if the dihedral angle between $s$ and $t$ is $\tfrac{\pi}{m_{st}}$, then we set $M_{st} = m_{st}$, and otherwise $M_{st} = \infty$. 
Let $W_P$ be the reflection group generated by the set of reflections $\sigma_s$ across the facets $s$ of $P$. The homomorphism $\sigma : W_{S,M} \to W_P$ defined by $\sigma(s) = \sigma_s$ is an isomorphism by Poincaré's polyhedron theorem, and so the Coxeter group $W_{S,M}$ obtained in this way is called a \emph{geometric reflection group}. Moreover, the nerve $N_W$ of the Coxeter group $W_{S,M}$ identifies with the simplicial complex $\mathcal{F}(\partial P^\ast) = \mathcal{F}(\partial P)^{op}$, where $P^\ast$ is the dual polytope of $P$, hence $N_W$ is PL homeomorphic to a $(d-1)$-dimensional sphere.
\end{remark}

Inspired by the previous remark, we make the following:

\begin{definition}
A Coxeter group $W$ is an \emph{abstract geometric reflection group of dimension $d$} if the geometric realization of the nerve $N_W$ of $W$ is PL homeomorphic to a $(d-1)$-dimensional sphere.
\end{definition}

 A \emph{spherical coset} of a Coxeter group $W$ is a coset of a spherical special subgroup of $W$. We denote by $W \mathcal{S}$ the set of all spherical cosets of $W$, i.e. 
$$  W \mathcal{S} = \bigsqcup_{T \in N_W \cup \{ \varnothing \} } W / W_T$$  
It is partially ordered by inclusion. To any poset $\mathcal{P}$ is associated an abstract simplicial complex 
$\mathrm{Flag}(\mathcal{P})$ consisting of all non-empty, finite, totally ordered subsets of $\mathcal{P}$. The geometric realization of $\mathrm{Flag}( W \mathcal{S})$ is called the \emph{Davis complex} $\Sigma_W$ of $W$.

\begin{example}
If $W$ is generated by the reflections in the sides of a square tile, then $N_W$ is a cyclic graph of length 4, and $\Sigma_W$ is the complex obtained by subdividing each tile of the infinite square grid into 8 isosceles right triangles, meeting at the center.
\end{example}

\begin{theorem}[Moussong \cite{moussong}, Stone \cite{MR0402648}, Davis--Januszkiewicz \cite{MR1131435}]\label{thm:merci_davis}
If a Coxeter group $W$ is an abstract geometric reflection group of dimension $d$, then the following hold:
\begin{itemize}
\item the Davis complex $\Sigma_W$ of $W$ admits a natural CAT(0)-metric;

\item the cell complex $\Sigma_W$ is homeomorphic to $\R^d$;

\item the CAT(0) boundary of $\Sigma_W$ is homeomorphic to $\mathbb{S}^{d-1}$.
\end{itemize} 
\end{theorem}

\begin{proof}[Sketch of the proof]  
For the reader’s convenience, we outline the important steps of the proof as in the book \cite{davis_book} and refer to the relevant theorems in \cite{davis_book}.

Moussong showed that the Davis complex $\Sigma_W$ equipped with its natural piecewise Euclidean structure is CAT(0) (see Theorem 12.3.3), and so $\Sigma_W$ is contractible (see also Theorem 8.2.13). Since (the geometric realization of) the nerve of $W$ is PL homeomorphic to a $(d-1)$-dimensional sphere, $\Sigma_W$ is a PL $d$-manifold (see Theorem 10.6.1). Therefore, the cell complex $\Sigma_W$ is a simply connected, nonpositively curved, piecewise Euclidean, PL manifold. By Stone’s theorem \cite{MR0402648}, $\Sigma_W$ is homeomorphic to $\R^d$ and moreover by Davis--Januszkiewicz’s theorem, the boundary of $\Sigma_W$ is a $(d-1)$-sphere (see Theorem I.8.4). 
\end{proof}

\section{Proof of Theorem \ref{thm:specifications}}

\begin{proposition}\label{prop:nerve_are_sphere}
Let $(W,S)$ be a Coxeter system with Coxeter diagram $\mathsf{G}_W$ in Tables \ref{Hexamples_dim4}, \ref{examples_dim4}, \ref{examples_dim5}, \ref{examples_dim6}, \ref{examples_dim7}, \ref{examples_dim8} and let $d$ be the rank of $W_S$ minus $2$. Then $W_S$ is an abstract geometric reflection group of dimension $d$.
\end{proposition}

\begin{proof}
Let $S_1$ be the set of white nodes of $\mathsf{G}_W$ and $S_2$ the set of black nodes of $\mathsf{G}_W$ so that $S = S_1 \cup S_2$. Note that the colors are just for reference and have no influence on the definition of the group. Using the tables in Appendix \ref{classi_diagram}, check that:
\begin{itemize}
\item the special subgroups $W_{S_1}$ and $W_{S_2}$ are Lannér;
\item the special subgroup $W_T$ with $T \subset S$ is spherical if and only if $S_1 \nsubset T$ and $S_2 \nsubset T$.
\end{itemize}

It therefore follows that the nerve $N_W$ of $W$ is isomorphic to the nerve of the Coxeter group $W_{S_1} \times W_{S_2}$ (as a poset). In other words, the nerve $N_W$ is isomorphic to the join of $N_1$ and $N_2$, where $N_i$ denotes the nerve of $W_{S_i}$ for each $i \in \{ 1,2 \}$. 

\medskip

Moreover, since $W_{S_i}$ is Lannér, its nerve $N_i$ is isomorphic to $\mathcal{F}( \partial \Delta_i)$ where $\Delta_i$ is the simplex of dimension $n_i = \sharp S_i - 1$, and in particular it is a simplicial complex whose geometric realization is a $(n_i - 1)$-dimensional sphere. Finally, since the join of the $n$-sphere and the $m$-sphere is the $(n+m+1)-$sphere, the geometric realization of the nerve $N_W$ is a $(d-1)$-dimensional sphere (recall that $d=n_1+n_2$).
\end{proof}

\begin{remark}
The proof of Proposition \ref{prop:nerve_are_sphere} is essentially the same as in Example 12.6.8 of \cite{davis_book} (see also Section 18 of Moussong \cite{moussong}).
\end{remark}

\begin{corollary}\label{cor:davis_details}
In the setting of Proposition \ref{prop:nerve_are_sphere}, the Coxeter group $W$ is Gromov-hyperbolic and the Gromov boundary of $W$ is a $(d-1)$-dimensional sphere.
\end{corollary}

\begin{proof}
First, we know from Theorem \ref{thm:merci_davis} that the Davis complex $\Sigma_W$ of $W$ is a CAT(0) space homeomorphic to $\R^{d}$ and that the CAT(0) boundary of $W$ is homeomorphic to $\mathbb{S}^{d-1}$. Second, it is easy to verify that the Coxeter group $W$ is Gromov-hyperbolic using Moussong's hyperbolicity criterion (see Theorem \ref{thm:moussong}). Finally, since the action of $W$ on $\Sigma_W$ is proper and cocompact, we have that $W$ is quasi-isometric to $\Sigma_W$ and the Gromov boundary of $W$ identifies to the CAT(0) boundary of $\Sigma_W$.
\end{proof}

\begin{theorem}[Moussong’s hyperbolicity criterion \cite{moussong}]\label{thm:moussong}
Assume that $W_S$ is a Coxeter group. Then the group $W_S$ is Gromov-hyperbolic if and only if $S$ does not contain any affine subset of rank $\geqslant 3$ nor two orthogonal non-spherical subsets.
\end{theorem}

In order to complete Theorem \ref{thm:specifications}, it remains to show the following:

\begin{proposition}\label{prop:signature}
In the setting of Proposition \ref{prop:nerve_are_sphere}, the signature of $W$ is $(d,2,0)$.
\end{proposition}

\begin{proof}
As seen in the proof of Proposition \ref{prop:nerve_are_sphere}, if $S_1$ (resp. $S_2$) denotes the set of all white (resp. black) nodes of $\mathsf{G}_W$, then the special subgroups $W_{S_1}$ and $W_{S_2}$ are Lannér, and for each $(s,t) \in S_1 \times S_2$, the special subgroup $W_{S \setminus \{s,t\}}$ is spherical. Thus the Coxeter group $W$ contains a spherical special subgroup of rank $d$ and a Lannér special subgroup. For example, if the Coxeter diagram $\mathsf{G}_W$ of $W$ is the diagram (A) in Figure \ref{fig:example}, then the diagram (B) corresponds to a spherical Coxeter group $B_3 \times I_2(p)$ (see Table \ref{spheri_diag}) and the diagram (C) corresponds to a Lannér Coxeter group in Table \ref{table:Lanner}.

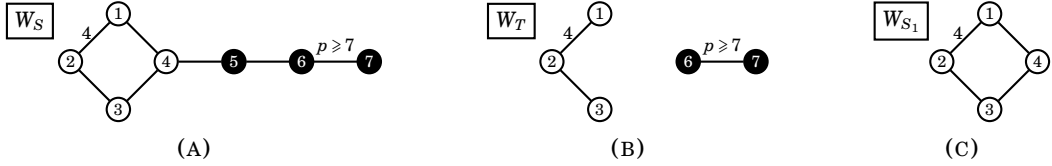
\begin{figure}[ht!]
\subfloat[]{
\begin{tikzpicture}[thick,scale=0.6, every node/.style={transform shape}]
\node[draw, inner sep=5pt] (1) at (-0.9,0.9) {\large $W_S$};
\node[draw,circle, inner sep=2pt, minimum size=2pt] (1) at (0,0) {{\small 2}};
\node[draw,circle, inner sep=2pt, minimum size=2pt] (2) at (1.06,1.06) {{\small 1}};
\node[draw,circle, inner sep=2pt, minimum size=2pt] (3) at (1.06,-1.06) {{\small 3}};
\node[draw,circle, inner sep=2pt, minimum size=2pt] (4) at (2.12,0) {{\small 4}};

\node[draw,circle,fill=black, inner sep=2pt, minimum size=2pt] (5) at (3.62,0) {{\small \textcolor{white}{5}}};
\node[draw,circle,fill=black, inner sep=2pt, minimum size=2pt] (6) at (5.12,0) {{\small \textcolor{white}{6}}};
\node[draw,circle,fill=black, inner sep=2pt, minimum size=2pt] (7) at (6.62,0) {{\small \textcolor{white}{7}}};

\draw (1) -- (3) node[below,near start] {};
\draw (2) -- (1) node[above,near end] {$4$};
\draw (2) -- (4);
\draw (3)--(4);
\draw (4) -- (5) node[above,near start] {};
\draw (5) -- (6) node[right,midway] {};
\draw (6) -- (7) node[above,midway] {$p \geqslant 7$};
\end{tikzpicture}
}
\quad\quad\quad
\subfloat[]{
\begin{tikzpicture}[thick,scale=0.6, every node/.style={transform shape}]
\node[draw, inner sep=5pt] (1) at (-0.9,0.9) {\large $W_T$};
\node[draw,circle, inner sep=2pt, minimum size=2pt] (1) at (0,0) {{\small 2}};
\node[draw,circle, inner sep=2pt, minimum size=2pt] (2) at (1.06,1.06) {{\small 1}};
\node[draw,circle, inner sep=2pt, minimum size=2pt] (3) at (1.06,-1.06) {{\small 3}};

\node[draw,circle,fill=black, inner sep=2pt, minimum size=2pt] (6) at (3.02,0) {{\small \textcolor{white}{6}}};
\node[draw,circle,fill=black, inner sep=2pt, minimum size=2pt] (7) at (4.52,0) {{\small \textcolor{white}{7}}};

\draw (1) -- (3) node[below,near start] {};
\draw (2) -- (1) node[above,near end] {$4$};

\draw (6) -- (7) node[above,midway] {$p \geqslant 7$};
\end{tikzpicture}
}
\quad\quad\quad
\subfloat[]{
\begin{tikzpicture}[thick,scale=0.6, every node/.style={transform shape}]
\node[draw, inner sep=5pt] (1) at (-0.9,0.9) {\large $W_{S_1}$};
\node[draw,circle, inner sep=2pt, minimum size=2pt] (1) at (0,0) {{\small 2}};
\node[draw,circle, inner sep=2pt, minimum size=2pt] (2) at (1.06,1.06) {{\small 1}};
\node[draw,circle, inner sep=2pt, minimum size=2pt] (3) at (1.06,-1.06) {{\small 3}};
\node[draw,circle, inner sep=2pt, minimum size=2pt] (4) at (2.12,0) {{\small 4}};

\draw (1) -- (3) node[below,near start] {};
\draw (2) -- (1) node[above,near end] {$4$};
\draw (2) -- (4);
\draw (3)--(4);

\end{tikzpicture}
}
\caption{A Coxeter group $W_{S}$ in Table \ref{examples_dim5} and two special subgroups $W_{T}$ and $W_{S_1}$ of $W_{S}$ with $S = \{ 1, 2, 3, 4\} \cup \{5, 6, 7\}$, $T = S \setminus \{ 4, 5\}$ and $S_1 = \{ 1, 2, 3, 4\}$}\label{fig:example} 
\end{figure}

Since the signature of any spherical (resp. Lannér) Coxeter group of rank $r$ is $(r,0,0)$ (resp. $(r-1,1,0)$), the signature $s_W$ of $W$ can only be $(d,2,0)$, $(d,1,1)$ or $(d+1,1,0)$, and so $s_W$ is determined by the sign of the determinant $\det(\mathsf{C}_W)$ of $\mathsf{C}_W$. More precisely, the signature $s_W$ is $(d,2,0)$, $(d,1,1)$ and $(d+1,1,0)$ if $\det(\mathsf{C}_W) > 0$, $ = 0$ and $< 0$, respectively.

\medskip

Now we check carefully the sign of $\det(\mathsf{C}_W)$ in the decreasing order of dimension $d$.

\medskip
 
In the case when $d = 8$ or $7$ (see Tables \ref{examples_dim8} or \ref{examples_dim7}), a simple computation shows that $\det(\mathsf{C}_W)$ is positive. For example, if the Coxeter diagram $\mathsf{G}_W$ of $W$ is:

\begin{center}
\begin{tikzpicture}[thick,scale=0.6, every node/.style={transform shape}]
\node[draw,circle, inner sep=2pt, minimum size=2pt] (1) at (-5,0){{\small 1}};
\node[draw,circle, inner sep=2pt, minimum size=2pt] (2) at (-3.5,0){{\small 2}};
\node[draw,circle, inner sep=2pt, minimum size=2pt] (3) at (-2,0){{\small 3}};
\node[draw,circle, inner sep=2pt, minimum size=2pt] (4) at (-0.5,0){{\small 4}};
\node[draw,circle, inner sep=2pt, minimum size=2pt] (5) at (1,0) {{\small 5}};

\node[draw,circle,fill=black, inner sep=2pt, minimum size=2pt] (4a) at (2.5,0) {{\small \textcolor{white}{6}}};
\node[draw,circle,fill=black, inner sep=2pt, minimum size=2pt] (5a) at (4,0) {{\small \textcolor{white}{7}}};
\node[draw,circle,fill=black, inner sep=2pt, minimum size=2pt] (6a) at (5.5,0) {{\small \textcolor{white}{8}}};
\node[draw,circle,fill=black, inner sep=2pt, minimum size=2pt] (7a) at (7,0) {{\small \textcolor{white}{9}}};
\node[draw,circle,fill=black, inner sep=0.5pt, minimum size=2pt] (8a) at (8.5,0) {{\small \textcolor{white}{10}}};

\draw (1)--(2) node[above,midway] {$5$};
\draw (2) -- (3);
\draw (3) -- (4);
\draw (4) -- (5);
\draw (4a) -- (5);
\draw (4a) -- (5a);
\draw (5a) -- (6a) node[right,midway] {};
\draw (6a) -- (7a) ;
\draw (7a) -- (8a) node[above,midway] {$5$};
\end{tikzpicture}
\end{center}
then the Cosine matrix $\mathsf{C}_W$ is:
{\footnotesize
\begin{displaymath}
\begin{pmatrix}
2 & -c_5 & 0 & 0 & 0 & 0 & 0 & 0 & 0 & 0 \\
-c_5 & 2 & -1 & 0 & 0 & 0 & 0 & 0 & 0 & 0 \\
0 & -1 & 2 & -1 & 0 & 0 & 0 & 0 & 0 & 0 \\
0 & 0 & -1 & 2 & -1 & 0 & 0 & 0 & 0 & 0 \\
0 & 0 & 0 & -1 & 2 & -1 & 0 & 0 & 0 & 0 \\
0 & 0 & 0 & 0 & -1 & 2 & -1 & 0 & 0 & 0 \\
0 & 0 & 0 & 0 & 0 & -1 & 2 & -1 & 0 & 0 \\
0 & 0 & 0 & 0 & 0 & 0 & -1 & 2 & -1 & 0 \\
0 & 0 & 0 & 0 & 0 & 0 & 0 & -1 & 2 & -c_5 \\
0 & 0 & 0 & 0 & 0 & 0 & 0 & 0 & -c_5 & 2 \\
\end{pmatrix}
\end{displaymath}
}
where $c_5 = 2 \cos( \frac{\pi}{5} )$, and the determinant of $\mathsf{C}_W$ is $\frac{1}{2}(25-11\sqrt{5}) \approx 0.201626$.

\medskip

In the case when $d = 6, 5, 4$ (see Tables \ref{examples_dim6}, \ref{examples_dim5}, \ref{examples_dim4}),  there is a one-parameter family $(W_p)_p$ or a two-parameter family $(W_{p,q})_{p,q}$ of Coxeter groups for each item of the Tables. The following Lemma \ref{lem:computation} shows that the determinant of the Cosine matrix increases when the parameter increases. Here, the (partial) order on the set of two parameters $(p,q)$ is given by:
$$
(p',q') \geqslant (p'',q'') \quad \Leftrightarrow \quad p' \geqslant p'' \quad \textrm{and} \quad q' \geqslant q''
$$

By an easy but long computation, we have that $\det(\mathsf{C}_{W_p})$ (resp. $\det(\mathsf{C}_{W_{p,q}})$) is positive for every minimal element $p$ (resp. $(p,q)$) in the set of parameters. For example, if the Coxeter diagram $\mathsf{G}_{W_p}$ of $W_p$ is:
\begin{center}
\begin{tikzpicture}[thick,scale=0.6, every node/.style={transform shape}]
\node[draw,circle,fill=black] (4) at (0,0) {};
\node[draw,circle] (3) at (-1,-1) {};
\node[draw,circle] (2) at (-1,1) {};
\node[draw,circle] (1) at (-2,0) {};

\node[draw,circle,fill=black] (5)  at (1.5,0)  {};
\node[draw,circle,fill=black] (6)  at (3,0)  {};
\draw (1) -- (3) node[below,near start] {$5$};
\draw (2) -- (1) node[above,near end] {$5$};

\draw (2) -- (4);
\draw (3)--(1);
\draw (4) -- (5) node[above,near start] {};
\draw (5) -- (6) node[above,midway] {$p \geqslant 11$};
\draw (3) -- (4) node[above,midway] {};
\end{tikzpicture}
\end{center}
then the determinant of $\mathsf{C}_{W_p}$ is equal to
 $-4 (3 + \sqrt{5}) + 8 (1 +\sqrt{5}) \cos\left( \frac{2 \pi}{p} \right) \approx 0.834557$ for $p = 11$. Remark that if $p=10$, then $\det(\mathsf{C}_{W_p}) = 0$ and so $W_p$ is the geometric reflection group of a compact hyperbolic $4$-polytope with $6$ facets (see Esselmann \cite{MR1396674}).
\end{proof}

\begin{lemma}\label{lem:computation}
If $(W_p)_{p}$ is a one-parameter family of Coxeter groups in Tables \ref{examples_dim4}, \ref{examples_dim5}, \ref{examples_dim6}, then $( \det(\mathsf{C}_{W_p}) )_p$ is an increasing sequence and its limit is a positive number. Similarly, if $(W_{p,q})_{p,q}$ is a two-parameter family of Coxeter groups in Table \ref{examples_dim4}, then $( \det(\mathsf{C}_{W_{p,q}}) )_{p,q}$ is also an increasing sequence and its limit is a positive number as $p$ and $q$ go to infinity.
\end{lemma}

\begin{proof}
We denote $W_p$ or $W_{p,q}$ by $W_{S,M}$ (simply $W_S$). Let $S_1$ (resp. $S_2$) be the set of all white (resp. black) nodes of the Coxeter diagram $\mathsf{G}_{W_S}$. 

\medskip

\textit{Assume first that there is a unique edge $\overline{st}$ between $S_1$ and $S_2$ in $\mathsf{G}_{W_S}$ such that $s \in S_1$ and $t \in S_2$}.\footnote{In practice, it means that $\mathsf{G}_{W_S}$ does not belong to the last row of Table \ref{examples_dim4}.}  
By Proposition 13 of Vinberg \cite{MR774946}, we have:
$$
\frac{\det(\mathsf{C}_{W_S})}{\det \left( \mathsf{C}_{W_{S \setminus \{s, t\}  }} \right)} = 
\frac{\det(\mathsf{C}_{W_{S_1}})}{\det \left(\mathsf{C}_{W_{S_1 \setminus \{s\}  }}\right) } \cdot
\frac{\det(\mathsf{C}_{W_{S_2}})}{\det \left(\mathsf{C}_{W_{S_2 \setminus \{t\}  }}\right)}  - 4 \cos^2 \left( \frac{\pi}{M_{st}} \right) 
$$
which is equivalent to:
$$
\det(\mathsf{C}_{W_S}) = \det(\mathsf{C}_{W_{S_1}}) \det(\mathsf{C}_{W_{S_2}}) - 4 \det \left( \mathsf{C}_{W_{S_1 \smallsetminus \{s\} }} \right)  \det \left( \mathsf{C}_{W_{S_2 \smallsetminus \{t\}}} \right)  \cos^2 \left( \frac{\pi}{M_{st}} \right) 
$$
by the following observation: for any two orthogonal subsets $T \perp U$ of $S$ (such as $S_1 \setminus \{s\}$ and $S_2 \setminus \{t\}$ here), 
  $$\det \left( \mathsf{C}_{W_{T \cup U}} \right) = \det \left(\mathsf{C}_{W_{T}}\right) \det \left(\mathsf{C}_{W_{U}}\right).$$ Moreover, it is easy to see the following:
\begin{itemize}
\item the sequence $p \mapsto  -\det(\mathsf{C}_{W_{S_2}})$ is positive and increasing to a positive number;
\item the sequence $p \mapsto \det \big( \mathsf{C}_{W_{S_2 \smallsetminus \{t\}}} \big)$ is positive and decreasing to zero;
\item in the case $W_S = W_{p}$, the numbers $-\det(\mathsf{C}_{W_{S_1}})$ and $ \det \big( \mathsf{C}_{W_{S_1 \smallsetminus \{s\}}} \big)$ are positive; and in the case $W_S = W_{p,q}$, the sequence $q \mapsto  -\det(\mathsf{C}_{W_{S_1}})$ is positive and increasing to a positive number, and the sequence $q \mapsto \det \big( \mathsf{C}_{W_{S_1 \smallsetminus \{s\}}} \big)$ is positive and decreasing to zero.
\end{itemize}
Hence, the sequence $p \mapsto \det(\mathsf{C}_{W_S})$ (or $(p,q) \mapsto \det(\mathsf{C}_{W_S})$) is increasing and converges to a positive number. 

\medskip

\textit{Now assume that there are two edges $\overline{rt}$ and $\overline{st}$ between $S_1$ and $S_2$ in $\mathsf{G}_{W_S}$ such that $r, s \in S_1 $ and $t \in S_2$}.\footnote{In other words, $\mathsf{G}_{W_S}$ belongs to the last row of Table \ref{examples_dim4}.} By Proposition 12 of Vinberg \cite{MR774946}, we have:
$$
\frac{\det(\mathsf{C}_{W_S})}{\det \left( \mathsf{C}_{W_{S \setminus \{t\}  }} \right)} - 2 = \left( \frac{\det \left(\mathsf{C}_{W_{ S_1 \cup \{t\} }} \right)}{\det \left( \mathsf{C}_{W_{S_1  }} \right)} - 2  \right) +  \left( \frac{\det(\mathsf{C}_{W_{S_2}})}{\det \left( \mathsf{C}_{W_{S_2 \setminus \{t\}  }} \right)} - 2  \right) 
$$
which is equivalent to
$$
\det \left( \mathsf{C}_{W_S} \right) =  \det(\mathsf{C}_{W_{S_1}}) \det(\mathsf{C}_{W_{S_2}}) 
+
\det \left( \mathsf{C}_{W_{S_1}} \right) \det \left( \mathsf{C}_{W_{S_2 \setminus \{t\} }} \right)  \left( \frac{\det \left( \mathsf{C}_{W_{S_1 \cup \{t\}}} \right)}{\det \left( \mathsf{C}_{W_{S_1}} \right)} - 2 \right) 
$$
by the above observation. Moreover, it is easy to see (cf. Table 2 of Esselmann \cite{MR1396674}) that
\begin{equation}\label{eq:two_values}
\frac{\det \left( \mathsf{C}_{W_{S_1 \cup \{t\}}} \right)}{\det \left( \mathsf{C}_{W_{S_1}} \right)} =
\begin{cases}
 \frac{5+2 \sqrt{5} + 3 \sqrt{2} + \sqrt{10}}{2} & \text{if } \mathsf{G}_{W_{S_1 \cup \{t\}}} \text{ is the left diagram in Figure \ref{Fig:Possible}}; \\
 3 + \sqrt{5}& \text{if } \mathsf{G}_{W_{S_1 \cup \{t\}}} \text{ is the right diagram in Figure \ref{Fig:Possible}.}
\end{cases}
\end{equation}

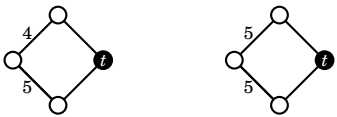
\begin{figure}[!ht]
\centering
\begin{tabular}{ccc}
\begin{tikzpicture}[thick,scale=0.6, every node/.style={transform shape}]
\node[draw,circle,fill=black,  inner sep=1pt, minimum size=1pt] (4) at (0,0) {\textcolor{white}{$t$} };
\node[draw,circle] (3) at (-1,-1) {};
\node[draw,circle] (2) at (-1,1) {};
\node[draw,circle] (1) at (-2,0) {};

\draw (1) -- (3) node[below,near start] {$5$};
\draw (2) -- (1) node[above,near end] {$4$};

\draw (2) -- (4);
\draw (3)--(1);
\draw (3) -- (4) node[above,midway] {};
\end{tikzpicture}
& \quad \quad \quad &
\begin{tikzpicture}[thick,scale=0.6, every node/.style={transform shape}]
\node[draw,circle,fill=black,  inner sep=1pt, minimum size=1pt] (4) at (0,0) {\textcolor{white}{$t$} };
\node[draw,circle] (3) at (-1,-1) {};
\node[draw,circle] (2) at (-1,1) {};
\node[draw,circle] (1) at (-2,0) {};

\draw (1) -- (3) node[below,near start] {$5$};
\draw (2) -- (1) node[above,near end] {$5$};

\draw (2) -- (4);
\draw (3)--(1);
\draw (3) -- (4) node[above,midway] {};
\end{tikzpicture}
\end{tabular}
\caption{Two possible Coxeter diagrams of $W_{S_1 \cup \{t\}}$ }
\label{Fig:Possible}
\end{figure}
Hence, the sequence $p \mapsto \det(\mathsf{C}_{W_S})$ is increasing and converges to a positive number, because the both values in Equation (\ref{eq:two_values}) are greater than $2$.
\end{proof}

Similarly, by an easy computation, we can prove:

\begin{proposition}\label{prop:signatureH}
Let $(W,S)$ be a Coxeter system with Coxeter diagram $\mathsf{G}_W$ in Table \ref{Hexamples_dim4}. Then the signature of $W_S$ is $(5,1,0)$.
\end{proposition}

\section{Tits representations and Convex cocompactness}

\subsection{Cartan matrices and Coxeter groups}

An $n \times n$ matrix $A = (A_{ij})_{i, j = 1, \dotsc, n}$ is called a \emph{Cartan matrix} if the following hold:
\begin{itemize}
\item for all $i = 1, \dotsc, n$, $A_{ii} = 2$, and for all $i \neq j$, \,$A_{ij} \leqslant 0$, and $A_{ij} =0 \Leftrightarrow A_{ji}=0$;
\item for all $i \neq j$,\, $A_{ij} A_{ji} \geqslant 4$ or $A_{ij} A_{ji}  = 4 \cos^{2} \left( \frac{\pi}{m_{ij}} \right)$ with some $m_{ij} \in \mathbb{N} \smallsetminus \{0, 1\} $.
\end{itemize}

A Cartan matrix $A$ is \emph{reducible} if (after a simultaneous permutation of rows and columns) $A$ is the direct sum of smaller matrices. Otherwise, $A$ is \emph{irreducible}. It is obvious that every Cartan matrix $A$ is the direct sum of irreducible matrices $A_1 \bigoplus \dotsm \bigoplus A_k$. Each irreducible matrix $A_i$, $i = 1, \dotsc, k$, is called a \emph{component} of $A$. If $x = (x_1, \dotsc, x_n)$ and $y = (y_1, \dotsc, y_n) \in \mathbb{R}^n$, we write $x > y$ if $x_i > y_i$ for every $i$, and $x \geqslant y$ if $x_i \geqslant y_i$ for every $i$.

\begin{proposition}[Theorem 3 of Vinberg \cite{bible}]\label{prop:cartan_type}
If $A$ is an irreducible Cartan matrix of size $n \times n$, then exactly one of the following holds:
\begin{itemize}
\item[$\mathrm{(P)}$] The matrix $A$ is nonsingular; for every vector $x \in \R^n$, if $Ax \geqslant 0$, then $x>0$ or $x=0$.

\item[$\mathrm{(Z)}$] The rank of $A$ is $n-1$; there exists a vector $y \in \R^n$ such that $y >0$ and $Ay =0$; for every vector $x \in \R^n$, if $Ax \geqslant 0$, then $Ax=0$.

\item[$\mathrm{(N)}$] There exists a vector $y \in \R^n$ such that $y >0$ and $Ay <0$; for every vector $x \in \R^n$, if $Ax \geqslant 0$ and $x \geqslant 0$, then $x=0$.
\end{itemize}
We say that $A$ is of \emph{positive, zero} or \emph{negative type} if $\mathrm{(P)}$, $\mathrm{(Z)}$ or $\mathrm{(N)}$ holds, respectively. 
\end{proposition}
A Cartan matrix $A$ is of \emph{positive} (resp. \emph{zero}) \emph{type} if every component of $A$ is of positive (resp. zero) type.

\begin{corollary}\label{cor:useful}
Let $A$ be a Cartan matrix (not necessarily irreducible). If there exists a vector $x > 0$ such that $A x =0$, then $A$ is of zero type.
\end{corollary}

A Cartan matrix $A = (A_{st})_{s,t \in S}$ is \emph{compatible} with a Coxeter group $W_{S,M}$ if for every $s \neq t \in S$, $A_{st} A_{ts}  = 4 \cos^2 \left( \frac{\pi}{M_{st}} \right)$ if $M_{st} < \infty$, and $A_{st} A_{ts} \geqslant 4$ otherwise. For any Coxeter group $W$, the Cosine matrix $\mathsf{C}_W$ of $W$ is in fact compatible with $W$.

\begin{proposition}[Propositions 22 and 23 of Vinberg \cite{bible}]\label{prop:usefull}
Assume that a Cartan matrix $A$ is compatible with a Coxeter group $W$. Then the following hold:
\begin{enumerate}
\item the matrix $A$ is of positive type if and only if $W$ is spherical;
\item if $A$ is of zero type, then $W$ is affine;
\item conversely, if $A=\mathsf{C}_W$ and $W$ is affine, then $A$ is of zero type.
\end{enumerate}
\end{proposition}

\subsection{Coxeter polytopes}

Let $V$ be a vector space over $\mathbb{R}$, and let $\mathbb{S}(V)$ be the projective sphere, i.e. the space of half-lines in $V$. We denote by $\mathrm{SL}^{\pm}(V)$ the group of automorphisms of $\mathbb{S}(V)$, i.e. $ \mathrm{SL}^{\pm}(V) = \{ X \in \mathrm{GL}(V) \mid \det (X) = \pm 1 \} $. The projective sphere $\mathbb{S}(V)$ and the group $\mathrm{SL}^{\pm}(V)$ are double covers of the projective space $\mathbb{P}(V)$ and the group  of projective transformations $\mathrm{PGL}(V)$, respectively.

\medskip

We denote by $\mathbb{S}$ the natural projection of $V \setminus \{0\}$ onto $\mathbb{S}(V)$. For any subset $U$ of $V$, $\mathbb{S}(U)$ denotes $\mathbb{S}(U \setminus \{ 0 \})$ for the simplicity of the notation. A subset $\mathcal{C}$ of $\mathbb{S}(V)$ is \emph{convex} if there exists a convex cone $U$ of $V$ such that $\mathcal{C} = \mathbb{S}(U)$, and $\mathcal{C}$ is \emph{properly convex} if in addition its closure $\overline{\mathcal{C}}$ does not contain a pair of antipodal points. In other words, $\mathcal{C}$ is properly convex if and only if $\overline{\mathcal{C}}$ is contained and convex in some affine chart. Note that if $\mathcal{C}$ is a properly convex subset of $\mathbb{S}(V)$, then the subgroup $\mathrm{SL}^{\pm}(\mathcal{C})$ of $\mathrm{SL}^{\pm}(V)$ preserving $\mathcal{C}$ is naturally isomorphic to a subgroup of $\mathrm{PGL}(V)$.

\medskip

A \emph{projective polytope} is a properly convex subset $P$ of $\mathbb{S}(V)$ such that $P$ has a non-empty interior and 
$ P = \bigcap_{i=1}^{n} \mathbb{S}( \{ x \in V \mid \alpha_i (x) \leqslant 0  \}    )$, where $\alpha_i$, $i=1, \dotsc, n$, are linear forms on $V$. We always assume that the set of linear forms is \emph{minimal}, i.e. none of the half spaces $\mathbb{S}( \{ x \in V \mid \alpha_i (x) \leqslant 0  \}    )$ contain the intersection of all the others. A \emph{projective reflection} is an element of $\mathrm{SL}^{\pm}(V)$ of order $2$ which is the identity on a hyperplane. Every projective reflection $\sigma$ can be written as:
$$  \sigma = \mathrm{Id} - \alpha \otimes b$$
where $\alpha$ is a linear form on $V$ and $b$ is a vector of $V$ such that $\alpha(b)=2$.

\medskip

A \emph{pre-Coxeter polytope} is a pair of a projective polytope $P$ of $\mathbb{S}(V)$:
$$ P = \bigcap_{s \in S} \mathbb{S}( \{ x \in V \mid \alpha_s (x) \leqslant 0  \}    )$$ 
and a set of projective reflections $(\sigma_s = \mathrm{Id} - \alpha_s \otimes b_s)_{s \in S}$ with $\alpha_s(b_s) =2$. A pre-Coxeter polytope is called a \emph{Coxeter polytope} if 
$$ \gamma \accentset{\circ}{P} \cap \accentset{\circ}{P} = \varnothing \quad \textrm{for every } \gamma \in \Gamma_P \backslash \{ \mathrm{Id} \}$$ 
where $\accentset{\circ}{P}$ is the interior of $P$ and $\Gamma_P$ is the subgroup of $\mathrm{SL}^{\pm}(V)$ generated by the reflections $(\sigma_{s})_{s \in S}$. If this is the case, then $\Gamma_P$ is a \emph{projective Coxeter group} generated by reflections $(\sigma_{s})_{s \in S}$. In \cite{bible}, Vinberg showed that a pre-Coxeter polytope $(\,P,\, (\sigma_s = \mathrm{Id} - \alpha_s \otimes b_s)_{s \in S} )$ is a Coxeter polytope if and only if the matrix $A = (\alpha_s(b_t))_{s,t \in S}$ is a Cartan matrix. We often denote the Coxeter polytope simply by $P$, and we call $A$ the \emph{Cartan matrix} of the Coxeter polytope $P$.

\medskip

 For any Coxeter polytope $P$, there is a unique Coxeter group, denoted $W_P$, compatible with the Cartan matrix of $P$. If $q$ is a subset of $P$, then we set $S_q = \{s \in S \mid q \subset \mathbb{S}(\mathrm{Ker} \alpha_s) \}$, where $\mathrm{Ker} \alpha_s$ is the kernel of $\alpha_s$.

\begin{theorem}[Tits, Chapter V of \cite{Bourbaki_group_456}, and Theorem 2 of Vinberg \cite{bible}]\label{theo_vinberg}
Let $P$ be a Coxeter polytope of $\mathbb{S}(V)$ with Coxeter group $W_P$, and let $\G_P$ be the group generated by the projective reflections $(\sigma_s)_{s \in S}$. Then the following hold:
\begin{enumerate}
\item the homomorphism $\sigma : W_P \rightarrow \mathrm{SL}^{\pm}(V)$ defined by $\sigma(s) =
\sigma_s$ is well-defined and is an isomorphism onto $\Gamma_P$, which is a discrete subgroup of $\mathrm{SL}^{\pm}(V)$;

\item the $\Gamma_P$-orbit of $P$ is a convex subset $\C_P$ of $\S(V)$;

\item if $\Omega_P$ is the interior of $\C_P$, then $\G_P$ acts properly discontinuously on $\Omega_P$;

\item for each point $x \in P$, the subgroup $\sigma (W_{S_x})$ is the stabilizer of $x$ in $\Gamma_P$;

\item an open face $f$ of $P$ lies in $\O_P$ if and only if the Coxeter group $W_{S_f}$ is finite.
\end{enumerate}
\end{theorem}

\subsection{Tits representations}\label{section:Tits}

To a Cartan matrix $A = (A_{st})_{s,t \in S}$ is associated a Coxeter simplex $\Delta_A$ of $\mathbb{S}(\mathbb{R}^S)$ as follows:
\begin{itemize}
\item for each $t \in S$, we set $\tilde{\alpha}_t = e_t^*$, where $(e_t^*)_{t \in S}$ is the canonical dual basis of $\R^S$;
\item for each $t \in S$, we take the unique vector $\tilde{b}_t \in \R^S$ such that $\tilde{\alpha}_s(\tilde{b}_t) = A_{st}$ for all $s \in S$;
\item the Coxeter simplex $\Delta_A$ is the pair of the projective simplex
 $\cap_{s \in S} \mathbb{S}( \{ x \in \mathbb{R}^S \mid \tilde{\alpha}_s(x) \leqslant 0 \})$ and the set of reflections $(\tilde{\sigma}_s = \mathrm{Id} - \tilde{\alpha}_s \otimes \tilde{b}_s)_{s \in S}$.
\end{itemize}

We let $W_A$ denote the unique Coxeter group compatible with $A$. Then Theorem \ref{theo_vinberg} provides the discrete faithful representation $\tilde{\sigma}_A : W_A \to  \mathrm{SL}^{\pm}(\mathbb{R}^S)$ whose image $\tilde{\Gamma}_A = \tilde{\sigma}_A(W_A)$ is the group generated by the reflections $(\tilde{\sigma}_s)_{s \in S}$. 

\medskip

 \textit{Assume now that none of the components of $A$ is of zero type.} We let $V_A$ denote the linear span of $(\tilde{b}_s)_{s \in S}$. For each $s \in S$, we set $\alpha_s = \tilde{\alpha}_s|_{V_A} $, the restriction of $\tilde{\alpha}_s$ to $V_A$, and $b_s = \tilde{b}_s$. By Proposition 13 of Vinberg \cite{bible}, the convex subset 
$$ P_A = \bigcap_{s \in S} \mathbb{S} (\{  x \in V_A \mid \alpha_s(x) \leqslant 0\}) $$ 
of $\mathbb{S}(V_A)$ has a non-empty interior. Moreover, since $\dim(V_A) = \mathrm{rank}(A)$, the subset $P_A$ is properly convex by  Proposition 18 of \cite{bible}. In other words, $P_A$ is a projective polytope of $\mathbb{S}(V_A)$ (cf. Corollary 1 of \cite{bible}). The projective polytope $P_A$ of $\mathbb{S}(V_A)$ together with the reflections $(\sigma_s = \mathrm{Id} - \alpha_s \otimes b_s)_{s \in S}$ is a Coxeter polytope, denoted again $P_A$, and  Theorem \ref{theo_vinberg} provides the discrete faithful representation $\sigma_A : W_A \to  \mathrm{SL}^{\pm}(V_A)$ whose image $\Gamma_A = \sigma_A(W_A)$ is the group generated by the reflections $(\sigma_s)_{s \in S}$. We denote by $P_A^*$ the convex hull of $(b_s)_{s \in S}$ in $\mathbb{S}(V_A)$, \ie
$$ P_A^* = \mathbb{S}(\{ x \in V_A \mid x = \sum_{s \in S} c_s b_s \, \textrm{ for some } c = (c_s)_{s \in S}  \geqslant 0 \}).$$

\medskip

 \textit{Assume in addition that $A$ is symmetric.} By Theorem 6 of Vinberg \cite{bible}, there exists a $\Gamma_A$-invariant scalar product $B_A$ on $V_A$ such that $B_A(b_s,b_t) = A_{st}$ for every $s, t \in S$, i.e. $\alpha_s = B_A(b_s, \cdot)$ for all $s \in S$.  Thus, the group $\Gamma_A$ is a subgroup of the orthogonal group $\mathrm{O}(B_A)$ of the non-degenerate scalar product $B_A$ on $V_A$, and preserves the \emph{negative open set of $A$}:
$$
O_A^{\light} = \mathbb{S}(\{ x \in V_A \mid B_A(x,x) < 0 \})
$$
Note that if the signature of $A$ is $(p,q,r)$, then $B_A$ is of signature $(p,q)$ and $\mathrm{O}(B_A)$ is naturally isomorphic to $\mathrm{O}_{p,q}(\mathbb{R})$. 

\medskip

For any Coxeter group $W_S$, the Cosine matrix $\mathsf{C}_W$ of $W_S$ is a symmetric Cartan matrix compatible with $W_S$, and so if $\mathsf{C}_W$ has no component of zero type, then there exists the representation $\sigma_{\mathsf{C}_W} : W_S \to \mathrm{O}(B_{\mathsf{C}_W})$, which we call the \emph{Tits--Vinberg representation} (simply \emph{Tits representation}) of $W_S$. 

\begin{remark}
The representation $\tilde{\sigma}_{\mathsf{C}_W} : W \to \mathrm{SL}^{\pm}(\mathbb{R}^S)$ exists without any condition on $\mathsf{C}_W$, and it is in fact conjugate to a representation of $W$ introduced by Tits (see Chapter V of \cite{Bourbaki_group_456}). However, if $\mathsf{C}_W$ is nonsingular, then $\mathsf{C}_W$ has no component of zero type, $V_{\mathsf{C}_W}=\mathbb{R}^S$ and the representation $\sigma_{\mathsf{C}_W} : W \to \mathrm{SL}^{\pm}(V_{\mathsf{C}_W})$ is conjugate to $\tilde{\sigma}_{\mathsf{C}_W} : W \to \mathrm{SL}^{\pm}(\mathbb{R}^S)$, i.e. the representations of $W$, $\sigma_{\mathsf{C}_W}$ and $\tilde{\sigma}_{\mathsf{C}_W}$, are essentially the same. 
\end{remark}

\subsection{Two Lemmas} 

In this section, we use the same notation as in Section \ref{section:Tits} and prove Theorem \ref{thm:dgk}, which is a generalization of Theorem 8.2 of Danciger--Guéritaud--Kassel \cite{DGK}. For any subset $U \subset S$, if $X$ is an $S \times S$ matrix or a vector of $\mathbb{R}^S$, then $X_U$ denotes the restriction of $X$ to $U \times U$ or $U$, respectively.

\begin{theorem}\label{thm:dgk}
Let $A = (A_{st})_{s,t \in S}$ be an irreducible symmetric Cartan matrix of signature $(p,q,r)$ for some $p, q \geqslant 1$, and let $W_A$ be the Coxeter group compatible with $A$. Assume that the following conditions are satisfied:
\begin{itemize}
\item[$(\mathsf{H}_0)$] for every subset $T$ of $S$, the matrix $A_T$ is \emph{not} of zero type;

\item[$(\mathsf{H}_-)$] there is \emph{no} pair of orthogonal subsets $T, U$ of $S$ such that $A_T$ and $A_U$ are of negative type.
\end{itemize}

Then the subgroup $\Gamma_A = \sigma_A(W_A)$ of $\mathrm{O}(B_A) \simeq \mathrm{O}_{p,q}(\mathbb{R})$ is $\mathbb{H}^{p,q-1}$-convex cocompact.
\end{theorem}

In the paper \cite{DGK}, the authors prove the Theorem in the case of right-angled Coxeter groups. The key lemmas, Lemmas 8.8 and 8.9 of \cite{DGK}, of the proof are generalized to arbitrary Coxeter groups in Lemmas \ref{lem:light} and \ref{lem:cocompact}, and the rest of the proof is the same as in the proof of Theorem 8.2 of \cite{DGK}.

\begin{lemma}\label{lem:light}
Let $A$ be a symmetric Cartan matrix and $O_A^{\light}$ the negative open set of $A$. Assume that $A$ has no component of zero type. Then the matrix $A$ satisfies $(\mathsf{H}_0)$ if and only if $P_A \cap P_A^\ast \subset O_A^{\light}$.
\end{lemma}

\begin{proof}
In the proof, the subscript $A$ is omitted from the notation for simplicity. Assume $A$ satisfies $(\mathsf{H}_0)$ and take $x\in P \cap P^\ast$. By definition, we have:
$$
B(b_s,x) \leqslant 0 \quad \textrm{for every } s \in S
\quad \textrm{and} \quad
x = \sum_{s \in S} c_s b_s \quad \textrm{for some } c = (c_s)_{s \in S}  \geqslant 0
$$
which imply that:
\begin{align*}
B(x,x) & = \sum_{s \in S} c_s \, B(b_s,x) \, \leqslant \, 0
\end{align*}
and the equality holds if and only if $c_s B(b_s,x) = 0$ for all $s \in S$.
Thus, $x \in \overline{O^\light}$. Suppose by contradiction that $x \in \partial O^\light$, i.e. $B(x,x)=0$. If $S^>$ denotes $\{ s \in S \mid  c_s > 0 \} $, then:
\begin{align*}
B(b_s,x)=0  & \quad \forall s \in S^> \\
\sum_{t \in S^>} c_t B(b_s,b_t)=0 &\quad  \forall s \in S^>  \\
A_{S^>} \cdot c_{S^>} =0 &
\end{align*}
Therefore, we conclude by Corollary \ref{cor:useful} that $A_{S^>}$ is of zero type, which contradicts the condition $(\mathsf{H}_0)$.

\medskip

To prove the converse we suppose by contradiction that there is a subset $U \subset S$ such that $A_U$ is of zero type. Then there exists a vector $c = (c_s)_{s \in U} > 0$ such that $A_U \cdot c = 0$ by Proposition \ref{prop:cartan_type}. Let $x=\sum_{s \in U} c_s b_s$. Now it is easy to see that $x\in P \cap P^\ast$ and $B(x,x)=0$: contradiction.
\end{proof}

\begin{lemma}\label{lem:cocompact}
In the setting of Theorem \ref{theo_vinberg}, if the Cartan matrix $A_P$ of $P$ (not necessarily symmetric) satisfies $(\mathsf{H}_0)$ and $(\mathsf{H}_-)$ then $P \cap P^\ast \subset \O_{P}$.
\end{lemma}

\begin{proof}
In the proof, the subscript $P$ is omitted from the notation for simplicity. Suppose that $x\in P \cap P^\ast$. In other words,
$$
\alpha_s(x) \leqslant 0 \quad \textrm{for every } s \in S
\quad \textrm{and} \quad
x = \sum_{s \in S} c_s b_s \quad \textrm{for some } c = (c_s)_{s \in S}  \geqslant 0.
$$
Let $S_x = \{ s \in S \mid \alpha_s(x) = 0\}  $. We aim to show that the Coxeter group $W_{S_x}$ is finite. Equivalently (Theorem \ref{theo_vinberg}), the stabilizer subgroup of $x$ in $\Gamma_P$ is finite. For this, we define the following subsets of $S$:
$$
S^> =\{ s \in S \mid c_s > 0 \} \quad\quad S_x^>= S_x \cap S^> \quad\quad S_x^0 = S_x \smallsetminus S_x^>
$$

First, we claim that $S_x^0 \perp S^>$. Indeed, by definition, we have $S_x^0 \cap S^> = \varnothing$, which implies that $A_{st} \leqslant 0$ for every $s \in S_x^0$ and every $t \in S^>$. Moreover, for each $s \in S_x$
\begin{align}
0= \alpha_s(x) = \sum_{t \in S^>} c_t \alpha_s(b_t) = \sum_{t \in S^>} c_t A_{st}.
\label{eq:lemma}
\end{align}
If $s \in S_x^0$, then each term of the right-hand sum in (\ref{eq:lemma}) is nonpositive, hence must be zero. Thus $A_{st} = 0$ for every $s \in S_x^0$ and every $t \in S^>$, which exactly means that $S_x^0  \perp S^>$ as claimed.

\medskip

Second, we claim that $A_{S_x^>}$ is of positive type. Indeed, if $s \in S_x^>$, then:
\begin{align*}
0= \alpha_s(x) = \sum_{t \in S} c_t A_{st} = \sum_{t \in S_x^>} c_t A_{st} + \sum_{t \notin S_x^>} c_t A_{st}
\end{align*}
which implies that $\sum_{t \in S_x^>} c_t A_{st} \geqslant 0$. Since $c_{S_x^>} > 0$ and $A_{S_x^>} \cdot c_{S_x^>} \geqslant 0$, we deduce from Proposition \ref{prop:cartan_type} that each component of $A_{S_x^>}$ is either of positive type or of zero type. Therefore, the condition $(\mathsf{H_0})$ implies that the matrix $A_{S_x^>}$ is of positive type as claimed.

\medskip

Third, we claim that the matrix $A_{S^>}$ has a component of negative type. Indeed, suppose by contradiction that every component of $A_{S^>}$ is not of negative type. Then the condition $(\mathsf{H_0})$ implies that $A_{S^>}$ is of positive type, and so by Lemma 13 of Vinberg \cite{bible}, we may assume without loss of generality that $A_{S^>}$ is a positive definite symmetric matrix, \ie for every non-zero vector $v = (v_s)_{s \in S^>}$, the scalar $v^T \cdot A_{S^>} \cdot v$ is positive, where $v^T$ denotes the transpose of $v$. However, this is impossible because: 
$$ c_{S^>}^T \cdot A_{S^>} \cdot c_{S^>} = \sum_{s, t \in S^>} c_s c_t A_{st} = \sum_{s \in S^>} c_s \alpha_s(x) \leqslant 0$$

\medskip

Fourth, we claim that $A_{S_x^0}$ is of positive type. Indeed, assume by contradiction that $A_{S_x^0}$ is not of positive type. Once again, the condition $(\mathsf{H_0})$ implies that $A_{S_x^0}$ has a component of negative type. This contradicts the condition $(\mathsf{H_-})$ because $S_x^0 \perp S^>$ and $A_{S^>}$ has a component of negative type.

\medskip

Finally, we claim that the Coxeter group $W_{S_x}$ is finite. Indeed, since $A_{S_x} = A_{S_x^>} \oplus A_{S_x^0}$ is of positive type, by Proposition \ref{prop:usefull} we have that the Coxeter group $W_{S_x}$ is spherical, i.e. it is finite as claimed.

\medskip

In conclusion, we deduce from the item (5) of Theorem \ref{theo_vinberg} that $x \in \Omega$.
\end{proof}

We now give a brief summary of the proof of Theorem \ref{thm:dgk} from \cite{DGK} for the reader’s convenience: 

\begin{proof}[Proof of Theorem \ref{thm:dgk}]
As in Theorem \ref{theo_vinberg}, we denote by $\Omega_{P_A}$ the interior of the $\Gamma_A$-orbit of $P_A$. Since $A$ is of negative type, the convex set $\Omega_{P_A}$ is properly convex by Lemma 15 of Vinberg \cite{bible}. Let $\mathcal{C}' \subset \mathbb{S}(\mathbb{R}^{p,q})$ be the $\Gamma_A$-orbit of $P_A \cap P_A^*$. By Lemmas \ref{lem:light} and \ref{lem:cocompact}, we have that $\mathcal{C}' \subset O_A^{\light} \cap \Omega_{P_A}$. In particular, the action of $\Gamma_A$ on $\mathcal{C}'$ is properly discontinuous, and cocompact since $P_A \cap P_A^*$ is a compact fundamental domain. Let $\mathcal{C}$ be the intersection of all $\Gamma_A$-translates of $P^*_A \cap \Omega_{P_A}$. By Lemma 8.7 of \cite{DGK}, the convex set $\mathcal{C}$ is non-empty. Since $\mathcal{C} \subset \mathcal{C}'$ and $\mathcal{C}$ is closed in $\Omega_{P_A}$, the action of $\Gamma_A$ on $\mathcal{C}$ is also properly discontinuous and cocompact. By Lemma 8.11 of \cite{DGK}, the set $\mathcal{C}$ is closed in $O_A^{\light}$. Using Lemmas 6.3 and 8.10 of \cite{DGK}, the authors complete the proof by showing that $\overline{\mathcal{C}} \setminus \mathcal{C}$ does not contain any non-trivial projective segment.
\end{proof}

\section{Topological actions of reflection groups}

In this section, we recall some facts about actions of reflection groups on manifolds (see Chapters 5 and 10 of \cite{davis_book}).

\medskip

A \emph{mirrored space over $S$} consists of a space $X$, an index set $S$ and a family of closed subspaces $(X_s)_{s\in S}$. We denote the mirrored space simply by $X$. For each $x\in X$, put $S_x=\{ s \in S \mid  x \in X_s \}$. 

\medskip

Suppose $X$ is a mirrored space over $S$ and $(W,S)$ is a Coxeter system. Define an equivalence relation $\sim$ on $W \times X$ by $(g,x) \sim (h,y)$ if and only if $x=y$ and $g^{-1}h \in W_{S_x}$. Give $W \times X$ the product topology and let $\Uc(W,X)$ denote the quotient space:
$$
\Uc(W,X) := W \times X / \sim
$$
We let $[g,x]$ denote the image of $(g,x)$ in $\Uc(W,X)$. The natural $W$-action on $W \times X$ is compatible with the equivalence relation, and so it descends to an action on $\Uc(W,X)$, i.e. $\g \cdot [g,x] = [\g g,x]$. The map $i:X \to \Uc(W,X)$ defined by $x \mapsto [1,x]$ is an embedding and we identify $X$ with its image under $i$. Note that $X$ is a \emph{strict fundamental domain} for the $W$-action on $\Uc(W,X)$, i.e. it intersects each $W$-orbit in exactly one point.

\medskip

An involution on a connected manifold $M$ is a \emph{topological reflection} if its fixed set separates $M$. Suppose that a group $W$ acts properly on a connected manifold $M$ and that it is generated by topological reflections. Let $R$ be the set of all topological reflections in $W$. For each $r \in R$, the fixed set of $r$, denoted $H_r$, is the \emph{wall} associated to $r$. The closure of a connected component of
 $ M \smallsetminus \bigcup_{r \in R} H_r$ is a \emph{chamber} of $W$. A wall $H_r$ is a \emph{wall of a chamber} $D$ if there is a point $x \in D \cap H_r$ such that $x$ belongs to no other wall. Fix a chamber $D$ and let
  $$S = \{ r \in R \mid H_r \textrm{ is a wall of } D\}.$$ 

\begin{theorem}[Proposition 10.1.5 of \cite{davis_book}]\label{thm:universal}
Suppose that $W$ acts properly on a connected manifold $M$ as a group generated by topological reflections. With the notation above, the following hold:
\begin{itemize}
\item the pair $(W,S)$ is a Coxeter system;
\item the natural $W$-equivariant map $\Uc(W,D) \to M$, induced by the inclusion $D \hookrightarrow M$, is a homeomorphism.
\end{itemize}
\end{theorem}

\begin{theorem}[Proposition 10.9.7 of \cite{davis_book}]\label{thm:must_be_reflection}
Suppose a Coxeter group $W$ acts faithfully, properly and cocompactly on a contractible manifold $M$. Then $W$ acts on $M$ as a group generated by topological reflections.
\end{theorem}

\begin{theorem}[Main Theorem of Charney--Davis \cite{charney_davis_when}]\label{thm:fund_gen_conj}
Suppose a Coxeter group $W$ admits a faithful, proper and cocompact action on a contractible manifold. If $S$ and $S'$ are two fundamental sets of generators for $W$, then there is a unique element $w \in W$ such that $S' = w S w^{-1}$.
\end{theorem}

\begin{lemma}\label{lem:must_be_cartan1}
Let $(W,S)$ be a Coxeter system, and let $\tau : W \to \mathrm{SL}^{\pm}(V)$ be a faithful representation. Suppose that $\tau(W)$ acts properly discontinuously and cocompactly on some properly convex open subset $\O$ of $\mathbb{S}(V)$. Then the following hold:
\begin{enumerate}
\item for each $s \in S$, the image $\tau(s)$ of $s$ is a projective reflection of $\mathbb{S}(V)$;
\item the group $\tau(W)$ is a projective Coxeter group generated by reflections $(\tau(s))_{s\in S}$.
\end{enumerate}
\end{lemma}

\begin{proof}
By Theorem \ref{thm:must_be_reflection}, there exists a set of generators $S' \subset W$ such that for each $s \in S'$ the image $\tau(s)$ of $s$ is a topological reflection, i.e. the fixed set of $\tau(s)$ separates $\O$. We claim that $\tau(s)$ is a projective reflection. Indeed, since $\tau(s)$ is an automorphism of $\mathbb{S}(V)$, if $\tau(s)$ is not a projective reflection, then the fixed set of $\tau(s)$ is of codimension $\geqslant 2$, hence does not separate $\O$: impossible.

\medskip

Let $R$ be the set of all reflections in $\tau(W)$. For each $r \in R$, we denote by $H_r$ the wall associated to $\tau(r)$, which is a hyperplane in $\O$. Fix a chamber $D$ of $\tau(W)$, which is a closed connected subset of $\O$, and let
  $$S'' = \{ r \in R \mid H_r \textrm{ is a wall of } D\}.$$

By Theorem \ref{thm:universal}, we have that $(W,S'')$ is a Coxeter system and that the natural $W$-equivariant map $\Uc(W,D) \to \O$ is a homeomorphism. Since the group $W$ acts faithfully, properly and cocompactly on $\O$, by Theorem \ref{thm:fund_gen_conj}, there is a unique element $w \in W$ such that $S = w S'' w^{-1}$. Thus, for each $s \in S$, the image $\tau(s)$ of $s$ is a projective reflection of $\mathbb{S}(V)$, and the strict fundamental domain $P := \tau(w) \cdot D$ is bounded by the fixed hyperplanes of the reflections $\tau(s)$ for $s \in S$. In other words, each automorphism $\tau(s)$ of $\mathbb{S}(V)$ can be written as $\tau(s) = \mathrm{Id}-\alpha_s  \otimes b_s$ for some linear form $\alpha_s \in V^*$ and some vector $b_s \in V$ so that $\alpha_s(b_s)=2$ and 
$$ P = \bigcap_{s \in S} \mathbb{S}(\{ x \in V \mid \alpha_s(x) \leqslant 0 \}). $$
Therefore, the pair of the polytope $P$ and the set of reflections $(\tau(s))_{s \in S}$ is a Coxeter polytope since $P$ is the (strict) fundamental domain for the $\tau(W)$-action on $\O$.
\end{proof}

\section{Proof of Theorems \ref{thm:tits_rep} and \ref{thm:tits_rep_H} }

In the following theorem, we use the same notation as in Section \ref{section:Tits}. 

\begin{theorem}\label{thm:fusion}
If $W$ is a Coxeter group with Coxeter diagram in Tables \ref{examples_dim4}, \ref{examples_dim5}, \ref{examples_dim6}, \ref{examples_dim7}, \ref{examples_dim8}, then the following hold:

\begin{itemize}
\item the group $W$ is Gromov-hyperbolic and its Gromov boundary is homeomorphic to $\mathbb{S}^{d-1}$;

\item the Cosine matrix $\mathsf{C}_W$ of $W$ satisfies the conditions $(\mathsf{H}_0)$ and $(\mathsf{H}_-)$;

\item the symmetric matrix $\mathsf{C}_W$ is of signature $(d,2,0)$;

\item the Tits representation $\sigma_{\mathsf{C}_W} :  W \to \mathrm{O}(B_{\mathsf{C}_W}) \simeq \mathrm{O}_{d,2}(\mathbb{R})$ is strictly GHC-regular. 
\end{itemize}
\end{theorem}

\begin{proof}
By Corollary \ref{cor:davis_details}, the Coxeter group $W$ is Gromov-hyperbolic and the Gromov boundary of $W$ is a $(d-1)$-dimensional sphere. By Proposition \ref{prop:usefull}, the Cosine matrix $\mathsf{C}_W$ of $W$ satisfies the hypotheses $(\mathsf{H}_0)$ and $(\mathsf{H}_-)$ because the Coxeter diagram $\mathsf{G}_W$ of $W$ has no edge of label $\infty$ and the Coxeter group $W$ satisfies Moussong's hyperbolicity criterion (Theorem \ref{thm:moussong}), namely $S$ does not contain any affine subset of rank at least $3$ nor two orthogonal non-spherical subsets. By Proposition \ref{prop:signature}, the signature of $W$ is $(d,2,0)$, and so the negative open set $O_{\mathsf{C}_W}^{\light}$ of $\mathsf{C}_W$ and the orthogonal group $\mathrm{O}(B_{\mathsf{C}_W})$ may be identified with $\mathrm{AdS}^{d,1}$ and $\mathrm{O}_{d,2}(\mathbb{R})$, respectively. By Theorem \ref{thm:dgk}, the subgroup $\sigma_{\mathsf{C}_W}(W)$ of $\mathrm{O}_{d,2}(\mathbb{R})$ is $\mathrm{AdS}^{d,1}$-convex cocompact, hence the representation $\sigma_{\mathsf{C}_W}$ is strictly GHC-regular because the Gromov boundary of $W$ is homeomorphic to $ \mathbb{S}^{d-1}$.
\end{proof}

Applying the same method as for Theorem \ref{thm:fusion}, we can prove:

\begin{theorem}\label{thm:Hfusion}
If $W$ is a Coxeter group with Coxeter diagram in Table \ref{Hexamples_dim4}, then the following hold:

\begin{itemize}
\item the group $W$ is Gromov-hyperbolic and its Gromov boundary is homeomorphic to $\mathbb{S}^{3}$;

\item the symmetric matrix $\mathsf{C}_W$ is of signature $(5,1,0)$;

\item the Tits representation $\sigma_{\mathsf{C}_W} :  W \to \mathrm{O}(B_{\mathsf{C}_W}) \simeq \mathrm{O}_{5,1}(\mathbb{R})$ is quasi-Fuchsian. 
\end{itemize}
\end{theorem}

Now, it only remains to show that every Coxeter group $W$ in Tables \ref{Hexamples_dim4}, \ref{examples_dim4}, \ref{examples_dim5}, \ref{examples_dim6}, \ref{examples_dim7}, \ref{examples_dim8} is not quasi-isometric to a uniform lattice of a semi-simple Lie group. First, the properties of $W$ being isomorphic to and quasi-isometric to a uniform lattice of a semi-simple Lie group respectively are almost the same due to the following beautiful theorem:

\begin{theorem}[Tukia \cite{tukia}, Pansu \cite{pansu}, Chow \cite{chow}, Kleiner--Leeb \cite{kleiner_leeb}, Eskin-Farb \cite{eskin_farb}, for an overview see Farb \cite{survey_farb} or Dru\c{t}u--Kapovich \cite{book_kapo_drutu}]
Let $\G$ be a finitely generated group and $X$ a symmetric space of noncompact type without Euclidean factor.\footnote{A \emph{symmetric space of noncompact type without Euclidean factor} means a manifold of the form $G/K$, where $G$ is a semi-simple Lie group without compact factors and $K$ is a maximal compact subgroup.} If the group $\G$ is quasi-isometric to $X$, then there exists a homomorphism $\tau : \G \to \mathrm{Isom}(X)$ with finite kernel and whose image is a uniform lattice of the isometry group $\mathrm{Isom}(X)$ of $X$.
\end{theorem}
Moreover, if in addition $\G$ is an irreducible \emph{large} Coxeter group $W$, \ie $W$ is not spherical nor affine, then the homomorphism $\tau : \G \to \mathrm{Isom}(X)$ must be injective:

\begin{proposition}[Assertion 2 of the proof of Proposition 4.3 of Paris \cite{paris}]\label{prop:paris} Let $W$ be an irreducible large Coxeter group. If $N$ is a normal subgroup of $W$ then $N$ is infinite.
\end{proposition}

Thus, an irreducible large Coxeter group $W$ is quasi-isometric to a uniform lattice of a semi-simple Lie group $G$ if and only if it is isomorphic to a uniform lattice of $G$. Even more, since $W$ is a Coxeter group, this problem reduces to being isomorphic or not to a uniform lattice of the isometry group of real hyperbolic space:

\begin{theorem}[Folklore]\label{thm_folklore} Let $W$ be an irreducible large Coxeter group and $X$ a symmetric space of noncompact type without Euclidean factor. If $W$ is isomorphic to a uniform lattice of $\mathrm{Isom}(X)$, then $X$ is a real hyperbolic space.
\end{theorem}

\begin{proof}
By Corollary 5.7 of Davis \cite{MR1600586}, the symmetric space $X$ is a product of real hyperbolic spaces. We aim to show that $X$ has only one factor. Let $G$ be the connected component of $\mathrm{Isom}(X)$ containing the identity, $\G_W$ the uniform lattice of $\mathrm{Isom}(X)$ isomorphic to $W$, and $\Gamma = \Gamma_W \cap G$. There exists a finite family $\{ G_i \}_{i \in I}$ of normal subgroups of $G$ such that (\emph{i}) $G = \prod_{i \in I} G_i$, (\emph{ii}) for each $i \in I$, $\G_i = G_i \cap \G$ is an irreducible lattice in $G_i$, and (\emph{iii}) $\prod_{i \in I} \G_i$ is a subgroup of finite index in $\Gamma$ (see e.g. Theorem 5.22 of Raghunathan \cite{MR0507234}). Since $W$ is irreducible and large, by Theorem 4.1 of Paris \cite{paris}, Proposition 8 of de Cornulier--de la Harpe \cite{cox_gp_indecomposable} or Theorem 1.2 of Qi \cite{MR2284573}, the group $W$ is \emph{strongly indecomposable}, i.e. for any finite-index subgroup $W'$ of $W$, there exists no non-trivial direct product decomposition of $W'$. Thus $\sharp I = 1$ and $\Gamma$ is an irreducible lattice in $G$. Finally, by Corollary 1.2 of Cooper--Long--Reid \cite{MR1626421} or Gonciulea \cite{MR1470091}, any finite-index subgroup of an infinite Coxeter group is not isomorphic to an irreducible lattice in a semi-simple Lie group of $\mathbb{R}$-rank $\geqslant 2$. Thus $X = \mathbb{H}^{d}$ for some $d \geqslant 2$ (see also the discussion after Corollary 1.3 of \cite{MR1626421}).
\end{proof}

\begin{remark}
In order to prove Corollary 1.2 of \cite{MR1626421}, the authors used the Margulis normal subgroup theorem \cite{margulis_normal_subgroup_thm} that any normal subgroup of an irreducible lattice in a semi-simple Lie group of $\mathbb{R}$-rank $\geqslant 2$ is either finite or finite index.
\end{remark}

\begin{remark}
Theorem 1.4 of Kleiner--Leeb \cite{MR1846203} shows that a finitely generated group quasi-isometric to $\mathbb{R}^n \times X$ with $X$ a symmetric space of noncompact type without Euclidean factor must have a finite-index subgroup containing $\Z^n$ inside its center. Any finite-index subgroup of an irreducible large Coxeter group has trivial center by Theorem 1.1 of Qi \cite{MR2284573}. Hence we may remove the hypothesis “without Euclidean factor” in Theorem \ref{thm_folklore} and change “semi-simple” to “reductive” in Theorems \ref{thm:tits_rep} and \ref{thm:tits_rep_H}. 
\end{remark}

\begin{proposition}\label{prop:not_lattice}
If $(W,S)$ is a Coxeter system with Coxeter diagram in Tables \ref{Hexamples_dim4}, \ref{examples_dim4}, \ref{examples_dim5}, \ref{examples_dim6}, \ref{examples_dim7}, \ref{examples_dim8}, then the group $W$ is not isomorphic to a uniform lattice of $\mathrm{Isom}(\mathbb{H}^e)$ for any integer $e \geqslant 1$.
\end{proposition}

\begin{proof}
Let $d$ be the rank of $W$ minus 2. Suppose by contradiction that there is a faithful representation $\tau:W \to \mathrm{Isom}(\mathbb{H}^e)$ whose image is a uniform lattice. By Corollary \ref{cor:davis_details}, we have $\partial W \simeq \S^{d-1}$. Fix $x_0 \in \mathbb{H}^e$. Since the orbit map $\varphi : W \to \mathbb{H}^e$ defined by $\g \mapsto \tau(\g) \cdot x_0$ is a quasi-isometry, it extends to a homeomorphism $\partial \varphi : \partial W \rightarrow \partial \mathbb{H}^e \simeq \S^{e-1}$, and so $d=e$. 

\medskip

Since $\tau(W)$ is a uniform lattice of $\mathrm{Isom}(\mathbb{H}^d)$, by Lemma \ref{lem:must_be_cartan1}, the group $\tau(W)$ is a projective Coxeter group generated by reflections $(\tau(s) = \mathrm{Id} - Q(v_s, \cdot) \otimes v_s )_{s\in S}$ for some $d+2$ vectors $v_s$ of $\R^{d,1}$ with $Q(v_s,v_s) = 2$, where $Q$ is the quadratic form on $\R^{d,1}$ defining $ \mathrm{PO}_{d,1}(\R) \simeq \mathrm{Isom}(\mathbb{H}^d)$. Therefore, the determinant of the symmetric Cartan matrix $\mathsf{C}_W = (Q(v_s,v_t))_{s,t \in S}$ must be zero. However, it is impossible since we have by Propositions \ref{prop:signature} and \ref{prop:signatureH} that the determinant of the Cosine matrix $\mathsf{C}_W$ of $W$ is non-zero.
\end{proof}

\begin{remark}
Proposition \ref{prop:not_lattice} (hence Lemma \ref{lem:must_be_cartan1}) is highly inspired by Proposition 3.1 of Benoist \cite{benoist_qi}, Example 12.6.8 of \cite{davis_book} and Examples 2.3.1 and 2.3.2 of Januszkiewicz--\'{S}wi{\k{a}}tkowski \cite{janu}.
\end{remark}

\section{An alternative proof of Moussong's criterion}\label{section:Moussong}

In the paper \cite{DGK}, the authors give an alternative proof of Moussong’s hyperbolicity criterion \cite{moussong} (see Theorem \ref{thm:moussong}) in the case of right-angled Coxeter groups. In this section, we prove the criterion for arbitrary Coxeter groups built on \cite{DGK}.

\medskip

Let $W_{S,M}$ be a Coxeter group. For any non-negative real number $\lambda$, the \emph{$\lambda$-Cosine matrix} $\mathsf{C}_W^\lambda$ of $W$ is a symmetric matrix whose entries are:
$$
(\mathsf{C}_W^\lambda)_{st} = -2\cos \bigg( \frac{\pi}{M_{st}}\bigg) \quad \textrm{if } M_{st} < \infty \quad \textrm{and}  \quad (\mathsf{C}_W^\lambda)_{st} = -2 (1+\lambda) \quad \textrm{otherwise}
$$
In the case $\lambda =0$, $\mathsf{C}_W^\lambda = \mathsf{C}_W$. When $\lambda > 0$, then by Proposition \ref{prop:usefull}, we have:
\begin{itemize}
\item the matrix $\mathsf{C}_W^\lambda$ satisfies $(\mathsf{H}_0)$ if and only if $S$ does not contain any affine subset of rank at least $3$;

\item  Assume that $\mathsf{C}_W^\lambda$ satisfies $(\mathsf{H}_0)$. Then $\mathsf{C}_W^\lambda$ satisfies $(\mathsf{H}_-)$ if and only if $S$ does not contain two orthogonal non-spherical subsets.
\end{itemize}

\begin{proof}[Proof of Moussong's criterion]
If $W$ is Gromov-hyperbolic, then it does not contain any subgroup isomorphic to $\mathbb{Z}^2$, and so it satisfies Moussong’s criterion. In other words, $S$ does not contain any affine subset of rank $\geqslant 3$ nor two orthogonal non-spherical subsets. Conversely, suppose that $W$ satisfies Moussong's criterion. We may assume without loss of generality that $W$ is infinite and irreducible. Then for any positive number $\lambda$, the $\lambda$-Cosine matrix $\mathsf{C}_W^\lambda$ of $W$ is an irreducible symmetric Cartan matrix of signature $(p,q,r)$ for some $p,q \geqslant 1$ satisfying the conditions $(\mathsf{H}_0)$ and $(\mathsf{H}_-)$. By Theorem \ref{thm:dgk}, the subgroup $\sigma_{\mathsf{C}_W^\lambda}(W)$ of $\mathrm{O}(B_{\mathsf{C}_W^\lambda})$ is $\mathbb{H}^{p,q-1}$-convex cocompact, hence $W$ is Gromov-hyperbolic by Theorem 1.7 of \cite{DGK}.
\end{proof}

\section{Proof of Theorems \ref{thm:disconnected} and \ref{thm:disconnected_H} }\label{section:disconnected}

In order to prove Theorems \ref{thm:disconnected} and \ref{thm:disconnected_H}, it is enough to exhibit two isolated points in the character variety of interest. To do this, we will consider a Coxeter diagram on $d+3$ nodes with one infinite edge label, such that the finite labels, when set to particular values, give rise to a uniform lattice $W_0$ of $\mathrm{PO}_{d,1}(\mathbb{R})$. In particular, the corresponding symmetric Cartan matrix $A_0$ has a two-dimensional kernel. We perturb the finite labels of $W_0$ to get a new Coxeter group $W$ so that the Gromov boundary of $W$ stays a $(d-1)$-sphere. Generically a symmetric Cartan matrix $A$ compatible with $W$ becomes invertible, but two choices for the entry of $A$ at the infinite label yield the result that $A$ has one-dimensional kernel, \ie an action on $\mathbb{H}^{d+1}$ or $\mathbb{H}^{d,1}$, depending on signature. These actions are rigid since there is only one infinite label to play with.

\begin{proof}[Proof of Theorems \ref{thm:disconnected} and \ref{thm:disconnected_H}]

Each item of Tables \ref{table:barbot2_dim4} or \ref{table:barbot2_dim6} corresponds to a one-parameter family $(W_p)_p$ or a two-parameter family $(W_{p,q})_{p,q}$ of Coxeter groups. We often denote $W_p$ or $W_{p,q}$ simply by $W$ or $W_S$. 

\medskip

\begin{table}[ht!]
\begin{tabular}{cc cc ccc}
\begin{tikzpicture}[thick,scale=0.7, every node/.style={transform shape}]
\node[draw,circle] (1) at (36-18:1.0514) {};
\node[draw,circle] (2) at (108-18:1.0514){};
\node[draw,circle] (3) at (180-18:1.0514){};
\node[draw,circle,fill=\Dshaded] (4) at (252-18:1.0514){};
\node[draw,circle,fill=\Lshaded] (5) at (324-18:1.0514){};
\draw (1)--(2)--(3)--(4)--(5)--(1);
\node[draw,circle,right=1cm of 1] (6) {};
\node[draw,circle,left=1cm of 3] (7) {};
\draw (1)--(6) node[above,midway] {$q \geqslant 7$};
\draw (3)--(7) node[above,midway] {$p \geqslant 7$};
\draw (4)--(5) node[below,midway] {$\infty$};
\draw (0.2,-2) node[]{$(p_0,q_0) := (10,10)$} ;
\end{tikzpicture}
&&
\begin{tikzpicture}[thick,scale=0.7, every node/.style={transform shape}]
\node[draw,circle] (1) at (36-18:1.0514) {};
\node[draw,circle] (2) at (108-18:1.0514){};
\node[draw,circle] (3) at (180-18:1.0514){};
\node[draw,circle,fill=\Dshaded] (4) at (252-18:1.0514){};
\node[draw,circle,fill=\Lshaded] (5) at (324-18:1.0514){};
\draw (1)--(2)--(3)--(4)--(5)--(1);
\node[draw,circle,right=1cm of 1] (6) {};
\node[draw,circle] (7) at ($(0,0)$) {};
\draw (1)--(6) node[above,midway] {$p \geqslant 7$};
\draw (4)--(5) node[below,midway] {$\infty$};
\draw (4)--(3);
\draw (5)--(1);
\draw (7)--(2);
\draw (7)--(3) node[midway] {$5$};
\draw (7)--(4) ;
\draw (0,-2) node[]{$p_0 := 10$} ;
\end{tikzpicture}
&&
\begin{tikzpicture}[thick,scale=0.7, every node/.style={transform shape}]
\node[draw,circle,fill=\Dshaded] (1) at (0,0) {};
\node[draw,circle] (2) at (0,2) {};
\node[draw,circle] (4) at (2,2) {};
\node[draw,circle] (3) at (1,2) {};
\node[draw,circle] (5) at (2,1) {};
\node[draw,circle,fill=\Lshaded] (6) at (2,0) {};
\draw (1)--(2)--(3)--(4)--(5)--(6)--(1);
\node[draw,circle] (7) at (3,1) {};
\draw (5)--(7) node[above,midway] {$q \geqslant 7$};
\draw (2)--(3) node[above,midway] {$p \geqslant 7$};
\draw (4)--(3) node[above,midway] {$4$};
\draw (1)--(2) node[left,midway] {$4$};
\draw (1)--(6) node[below,midway] {$\infty$};
\draw (1.3,-1) node[]{$(p_0,q_0) := (8,8)$} ;

\end{tikzpicture}
\\
\begin{tikzpicture}[thick,scale=0.7, every node/.style={transform shape}]
\node[draw,circle] (1) at (36-18:1.0514) {};
\node[draw,circle] (2) at (108-18:1.0514){};
\node[draw,circle] (3) at (180-18:1.0514){};
\node[draw,circle,fill=\Dshaded] (4) at (252-18:1.0514){};
\node[draw,circle,fill=\Lshaded] (5) at (324-18:1.0514){};
\draw (1)--(2)--(3)--(4)--(5)--(1);
\node[draw,circle,right=1cm of 1] (6) {};
\node[draw,circle,left=1cm of 3] (7) {};
\draw (1)--(6) node[above,midway] {$q \geqslant 7$};
\draw (3)--(7) node[above,midway] {$p \geqslant 7$};
\draw (4)--(5) node[below,midway] {$\infty$};
\draw (4)--(3) node[left,midway] {$4$};
\draw (5)--(1) node[right,midway] {$4$};
\draw (0.1,-2) node[]{$(p_0,q_0) := (8,8)$} ;
\end{tikzpicture}
&&
\begin{tikzpicture}[thick,scale=0.7, every node/.style={transform shape}]
\node[draw,circle] (1) at (36-18:1.0514) {};
\node[draw,circle] (2) at (108-18:1.0514){};
\node[draw,circle] (3) at (180-18:1.0514){};
\node[draw,circle,fill=\Dshaded] (4) at (252-18:1.0514){};
\node[draw,circle,fill=\Lshaded] (5) at (324-18:1.0514){};
\draw (1)--(2)--(3)--(4)--(5)--(1);
\node[draw,circle,right=1cm of 1] (6) {};
\node[draw,circle] (7) at ($(0,0)$) {};
\draw (1)--(6) node[above,midway] {$p \geqslant 7$};
\draw (4)--(5) node[below,midway] {$\infty$};
\draw (4)--(3) node[left,midway] {$4$};
\draw (5)--(1) node[right,midway] {$4$};
\draw (7)--(2);
\draw (7)--(3) node[midway] {$4$};
\draw (7)--(4) node[right,midway] {$4$};
\draw (0,-2) node[]{$ p_0 := 8$} ;
\end{tikzpicture}
&&
\begin{tikzpicture}[thick,scale=0.7, every node/.style={transform shape}]
\node[draw,circle,fill=\Dshaded] (1) at (0,0) {};
\node[draw,circle] (2) at (0,2) {};
\node[draw,circle] (3) at (1,2) {};
\node[draw,circle] (4) at (2,2) {};
\node[draw,circle] (5) at (2,1) {};
\node[draw,circle,fill=\Lshaded] (6) at (2,0) {};
\draw (1)--(2)--(3)--(4)--(5)--(6)--(1);
\node[draw,circle] (7) at (3,1) {};
\draw (5)--(7) node[midway] {$4$};
\draw (2)--(3) node[above,midway] {$p \geqslant 7$};
\draw (4)--(3) node[above,midway] {$4$};
\draw (1)--(2) node[left,midway] {$4$};
\draw (1)--(6) node[below,midway] {$\infty$};
\draw (4)--(7);
\draw (6)--(7);
\draw (1,-1) node[]{$p_0 := 8$} ;
\end{tikzpicture}
\end{tabular}
\caption{Abstract geometric reflection groups of dimension $d=4$}
\label{table:barbot2_dim4}
\begin{tikzpicture}[thick,scale=0.7, every node/.style={transform shape}]
\node[draw,circle] (1) at (36-18:1.0514) {};
\node[draw,circle] (2) at (108-18:1.0514){};
\node[draw,circle] (3) at (180-18:1.0514){};
\node[draw,circle,fill=\Dshaded] (4) at (252-18:1.0514){};
\node[draw,circle,fill=\Lshaded] (5) at (324-18:1.0514){};
\draw (1)--(2)--(3)--(4)--(5)--(1);
\node[draw,circle,right=1cm of 1] (6) {};
\node[draw,circle,right=1cm of 6] (7) {};
\node[draw,circle,right=1cm of 7] (8) {};
\draw (1)--(6)--(7)--(8);
\draw (7)--(8) node[above,midway] {$5$};
\node[draw,circle,left=1cm of 3] (9) {};
\draw (3)--(9) node[below,midway] {$p \geqslant 7$};
\draw (4)--(5) node[below,midway] {$\infty$};
\draw (0,-2) node[]{$p_0 := 10$} ;
\end{tikzpicture}
\caption{Abstract geometric reflection groups of dimension $d=6$}
\label{table:barbot2_dim6}
\end{table}

Let $d$ be the rank of $W$ minus $3$. It is easy to verify that the Coxeter group $W$ is Gromov-hyperbolic using Moussong’s hyperbolicity criterion (see Theorem \ref{thm:moussong}). In \cite{tumarkin_n+3}, Tumarkin showed that if $(p,q)=(p_0,q_0)$ (resp. $p=p_0$), then $W_{p,q}$ (resp. $W_{p}$) is the geometric reflection group of a compact hyperbolic $d$-polytope. Note that if $W_0$ and $W_1$ are two Coxeter groups such that ($i$) the underlying graphs of $\mathsf{G}_{W_0}$ and $\mathsf{G}_{W_1}$ are equal; ($ii$) the Coxeter diagram $\mathsf{G}_{W_1}$ is obtained from $\mathsf{G}_{W_0}$ by changing from a label $m_0 \geqslant 7$ on a fixed edge to another label $m_1 \geqslant 7$, then the nerves of $W_0$ and $W_1$ are isomorphic. Thus, the nerve of any Coxeter group $W_{p,q}$ (resp. $W_p$) is isomorphic to the nerve of $W_{p_0,q_0}$ (resp. $W_{p_0}$), and so the Coxeter group $W$ is an abstract geometric reflection group of dimension $d$. By Theorem \ref{thm:merci_davis}, the Davis complex of $W$ and the Gromov boundary of $W$ are homeomorphic to $\mathbb{R}^d$ and $\mathbb{S}^{d-1}$, respectively.

\medskip

We denote by $S_1$ (resp. $S_2$) the complement of the light-shaded (resp. dark-shaded) node in $S$. For example, if $W_S$ corresponds to the top left diagram in Table \ref{table:barbot2_dim4}, then $W_{S_1}$ (resp. $W_{S_2}$) corresponds to the diagram (A) (resp. (B)) in Figure \ref{fig:discon_specific}. 

\medskip

\begin{figure}[ht!]
\subfloat[]{
\begin{tikzpicture}[thick,scale=0.7, every node/.style={transform shape}]
\node[draw,circle] (1) at (36-18:1.0514) {};
\node[draw,circle] (2) at (108-18:1.0514){};
\node[draw,circle] (3) at (180-18:1.0514){};
\node[draw,circle,fill=\Dshaded] (4) at (252-18:1.0514){};
\draw (1)--(2)--(3)--(4);
\node[draw,circle,right=1cm of 1] (6) {};
\node[draw,circle,left=1cm of 3] (7) {};
\draw (1)--(6) node[above,midway] {$q \geqslant 7$};
\draw (3)--(7) node[above,midway] {$p \geqslant 7$};
\end{tikzpicture}
}
\quad\quad\quad
\subfloat[]{
\begin{tikzpicture}[thick,scale=0.7, every node/.style={transform shape}]
\node[draw,circle] (1) at (36-18:1.0514) {};
\node[draw,circle] (2) at (108-18:1.0514){};
\node[draw,circle] (3) at (180-18:1.0514){};
\node[draw,circle,fill=\Lshaded] (5) at (324-18:1.0514){};
\draw (1)--(2)--(3);
\draw (5)--(1);
\node[draw,circle,right=1cm of 1] (6) {};
\node[draw,circle,left=1cm of 3] (7) {};
\draw (1)--(6) node[above,midway] {$q \geqslant 7$};
\draw (3)--(7) node[above,midway] {$p \geqslant 7$};
\end{tikzpicture}
}
\caption{Two special subgroups of the Coxeter group $W_S$}\label{fig:discon_specific}
\end{figure}
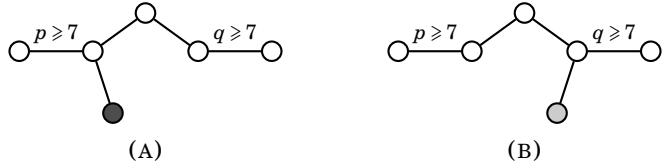

By the same reasoning as in the proof of Proposition \ref{prop:signature}, the signature $s_{W_{S_i}}$ of $W_{S_i}$, $i=1,2$, is determined by the sign of the determinant of $\mathsf{C}_{W_{S_i}}$. More precisely, the signature $s_{W_{S_i}}$ is $(d,2,0)$, $(d,1,1)$ and $(d+1,1,0)$ when $\det(\mathsf{C}_{W_{S_i}}) > 0$, $ = 0$ and $< 0$, respectively.

\medskip

We set $f(\lambda) = \det(\mathsf{C}_W^\lambda)$, which is a quadratic polynomial of $\lambda$:
$$ f(\lambda) = a_2 \lambda^2 + a_1 \lambda + a_0$$ 
It is not difficult (but important) to see that the discriminant $\delta$ of $f(\lambda)$ is:
$$  \delta = 16 \, \det(\mathsf{C}_{W_{S_1}}) \det(\mathsf{C}_{W_{S_2}}) $$
Thus, the polynomial $f(\lambda)$ has two distinct positive real roots if and only if
\begin{equation}\label{inequ:disconnected}
 \det(\mathsf{C}_{W_{S_1}}) \det(\mathsf{C}_{W_{S_2}})  > 0, \quad a_1 a_2 < 0 \quad \textrm{and} \quad a_0 a_2 >0.
\end{equation}

\begin{lemma}
Let $W$ be a Coxeter group with Coxeter diagram in Tables \ref{table:barbot2_dim4} or \ref{table:barbot2_dim6}. Then $W$ satisfies the condition (\ref{inequ:disconnected}) if and only if the Coxeter diagram of $W$ lies in Tables \ref{table:barbot2_not-poincare_hyp_dim4}, \ref{table:barbot2_not-poincare_hyp_dim6}, \ref{table:barbot2_Quasi_Fuchsian_dim4}, \ref{table:barbot2_Quasi_Fuchsian_dim6}.
\end{lemma}

\begin{proof}
We only prove it for the Coxeter group $W_{p,q}$ in the top left item of Table \ref{table:barbot2_dim4}; the argument is similar for other Coxeter groups. 

\medskip

We set $c_m := \cos(\frac{2\pi}{m})$, and let $x = c_p$ and $y = c_q$. Note that $x, y > \frac{1}{2}$ because $p, q \geqslant 7$. By a simple computation, we have that the polynomial $f(\lambda)$ is: 
\begin{center}
\bgroup
\def\arraystretch{1.3}
\begin{tabular}{c@{} c@{} c@{} c@{} c}
\multicolumn{5}{c}{$ f(\lambda)=a_2(x,y) \lambda^2 +  a_1(x,y) \lambda + a_0(x,y)$} \\
 $a_0 (x,y) $ & \, $=$ & \, $8(2 x + 2 y - 3 )$ &  & \\
 $a_1 (x,y) $ & \, $=$ & \, $-16 ( 2 x - 1 ) ( 2 y - 1 ) $ & \, $<$ & 0\\
 $a_2 (x,y) $ & \, $=$ & \, $8 \left(1 - (2 x - 1 )(2 y - 1 ) \right)$ & \, $>$ & 0\\
\end{tabular}
\egroup
\end{center}
and the discriminant $\delta(x,y)$ of $f(\lambda)$ is:
\begin{center}
\bgroup
\def\arraystretch{1.3}
\begin{tabular}{c}
$\delta(x,y) = 16 \det(\mathsf{C}_{W_{S_1}}) \det(\mathsf{C}_{W_{S_2}})$ \\
$\det(\mathsf{C}_{W_{S_1}}) = 4(4 x y - 2 x - 1) \quad \textrm{and} \quad \det(\mathsf{C}_{W_{S_2}}) = 4(4 x y - 2 y - 1)$
\end{tabular}
\egroup
\end{center}
In Figure \ref{figure:region}, the curves $\det(\mathsf{C}_{W_{S_1}}) = 0$, $\det(\mathsf{C}_{W_{S_2}}) = 0$ and $a_0(x,y) = 0$ are drawn in red, blue and green, respectively. 

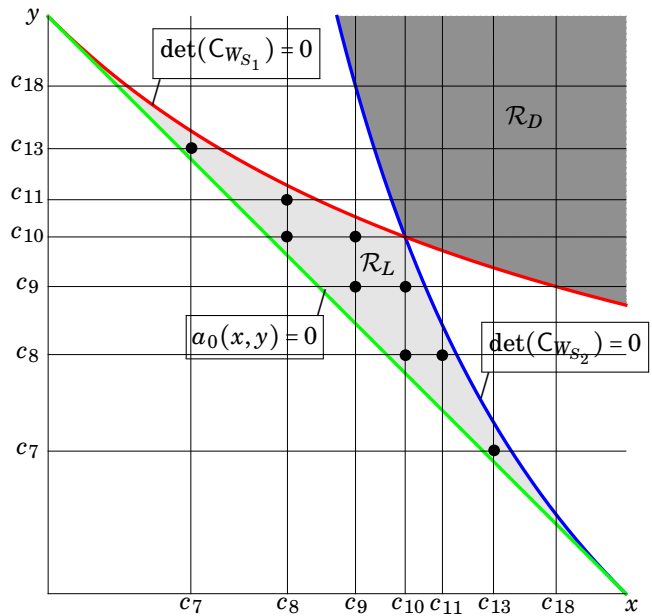
\begin{figure}[ht!]
\centering
\labellist
\footnotesize \hair 2pt
\pinlabel $c_7$ at 93 0
\pinlabel $c_8$ at 151 0
\pinlabel $c_9$ at 192 0
\pinlabel $c_{10}$ at 223 0
\pinlabel $c_{11}$ at 246 0
\pinlabel $c_{13}$ at 275 0
\pinlabel $c_{18}$ at 313 0
\pinlabel $x$ at 357 0
\pinlabel $c_7$ at -5 93
\pinlabel $c_8$ at -5 151
\pinlabel $c_9$ at -5 192
\pinlabel $c_{10}$ at -5 223
\pinlabel $c_{11}$ at -5 246 
\pinlabel $c_{13}$ at -5 275 
\pinlabel $c_{18}$ at -5 313 
\pinlabel $y$ at -2 353
\pinlabel $\mathcal{R}_D$ at 293 293
\pinlabel $\mathcal{R}_L$ at 205 205
\pinlabel $/$ at 72 308
\pinlabel \framebox[1.1\width]{$\det(\mathsf{C}_{W_{S_1}})=0$} at 120 331
\pinlabel $/$ at 269 132
\pinlabel \framebox[1.1\width]{$\det(\mathsf{C}_{W_{S_2}})=0$} at 317 155
\pinlabel $/$ at 171 182
\pinlabel \framebox[1.1\width]{$a_0(x,y)=0$} at 131 161
\endlabellist
\includegraphics[height=8cm]{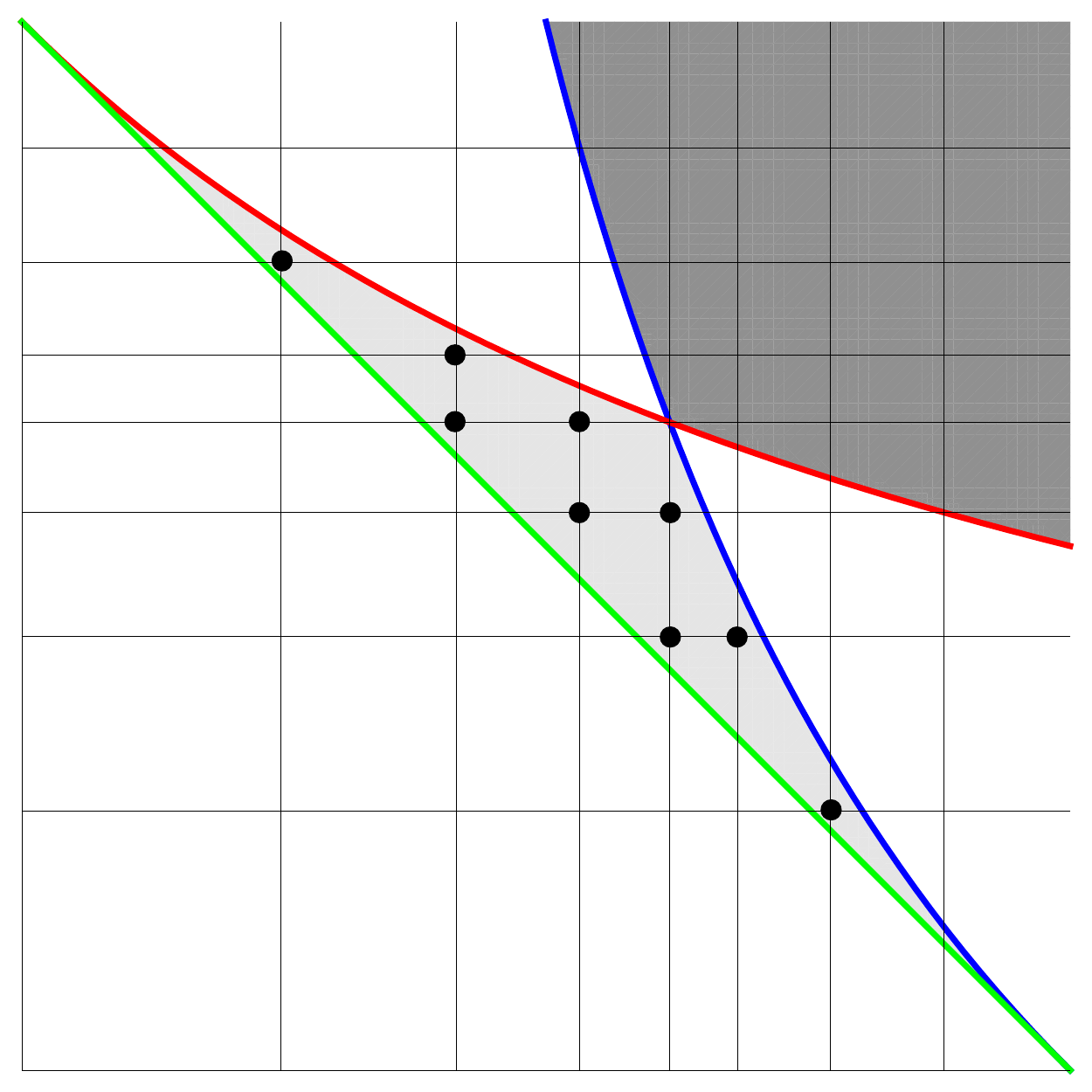}
\caption{Two open regions $\mathcal{R}_D$ and $\mathcal{R}_L$}\label{figure:region}
 \end{figure}

\medskip

Observe that the Coxeter group $W_{p,q}$ satisfies the condition (\ref{inequ:disconnected}) if and only if $(x,y) = (c_p,c_q)$ lies in the dark-shaded region $\mathcal{R}_D$ or the light-shaded region $\mathcal{R}_L$. Finally, it follows from Figure \ref{figure:region} that:
\begin{center}
\bgroup
\def\arraystretch{1.3}
\begin{tabular}{ccc}
$ (c_p,c_q) \in \mathcal{R}_D$ & $ \Leftrightarrow  $ & $p, q \geqslant 9 \textrm{ and }  \{p,q\} \neq \{ 9, 9 \}, \{ 9, 10 \}, \dotsc, \{ 9, 18 \}, \{10,10\}; $ \\
$(c_p,c_q) \in \mathcal{R}_L$ & $\Leftrightarrow $  & $  \{p,q\} = \{ 7, 13 \}, \{ 8, 10 \}, \{ 8, 11 \}, \{9,9\}, \{9,10\} $ 
\end{tabular}
\egroup
\end{center}
For example, the point $(c_p,c_q)$ lies in $\mathcal{R}_L$ if and only if it corresponds to one of the black dots in Figure \ref{figure:region}.
\end{proof}

\emph{Assume now that the Coxeter diagram of $W_S$ lies in Tables \ref{table:barbot2_not-poincare_hyp_dim4}, \ref{table:barbot2_not-poincare_hyp_dim6}, \ref{table:barbot2_Quasi_Fuchsian_dim4}, \ref{table:barbot2_Quasi_Fuchsian_dim6}.} 
Then the polynomial $f(\lambda) = \det(\mathsf{C}_W^\lambda)$ has two distinct positive roots, namely $\lambda_1$ and $\lambda_2$. Note that the matrices $\mathsf{C}_{W_{S_1}}$ and $\mathsf{C}_{W_{S_2}}$ have the same signature because their determinants have the same sign. If the Cosine matrix $\mathsf{C}_{W_{S_1}}$ (hence $\mathsf{C}_{W_{S_2}}$) has the signature $(m,n,0)$ with $(m,n)=(d,2)$ or $(d+1,1)$, then the signature of $\mathsf{C}_i := \mathsf{C}^{\lambda_i}_{W}$, $i=1,2$, is $(m,n,1)$. Thus, Theorem \ref{thm:dgk} provides $\mathbb{H}^{m,n-1}$-convex cocompact representations $\sigma_{\mathsf{C}_i} : W \to \mathrm{PO}(B_{\mathsf{C}_i}) \simeq \mathrm{PO}_{m,n}(\mathbb{R})$ for $i = 1, 2$. For example, in the case where $W = W_{p,q}$ in the top left item of Table \ref{table:barbot2_dim4} and the point $(c_p,c_q)$ lies in $\mathcal{R}_D$ (resp. $\mathcal{R}_L$), the Coxeter group $W$ admits two convex cocompact actions on $\mathbb{H}^{4,1}$ (resp.  $\mathbb{H}^{5}$).

\medskip

We let $[\rho]$ denote the equivalence class of the representation $\rho$, i.e. an element of the $\mathrm{PO}_{m,n}(\mathbb{R})$-character variety of $W$. Since $\mathsf{C}_1 \neq \mathsf{C}_2$, the characters $[\sigma_{\mathsf{C}_1}]$ and $[\sigma_{\mathsf{C}_2}]$ are different. Finally, we claim that for each $i=1,2$, the connected component $\chi_i$ of the $\mathrm{PO}_{m,n}(\mathbb{R})$-character variety of $W$ containing $[\sigma_{\mathsf{C}_i}]$ is a singleton. Indeed, suppose that $\tau_1 \in \chi_i$. Then there exists a continuous path $\tau_t : [0,1] \to \chi_i$ from $\tau_0 = \sigma_{\mathsf{C}_i}$ to $\tau_1$. This implies that for each torsion element $\gamma \in W$, the eigenvalues of $\tau_t(\gamma)$ do not change, i.e. the image $\tau_t(\gamma)$ is conjugate to $\tau_0(\gamma) = \sigma_{\mathsf{C}_i}(\gamma)$. In particular, for every $s \in S$, $\tau_t(s)$ is a projective reflection in $\mathrm{PO}_{m,n}(\mathbb{R})$, and there exists a continuous family of symmetric Cartan matrices $(A_t)_{t \in [0,1]}$ with $\sigma_{A_t} = \tau_t$. Hence we conclude that there is a continuous function $\mu_t : [0,1] \to \mathbb{R}$ with $\mathsf{C}^{\mu_t}_W = A_t$ because the diagram $\mathsf{G}_W$ has only one edge of label $\infty$. The matrix $A_t$ must be of rank $\leqslant m + n = d+2$. In other words, $\det( \mathsf{C}^{\mu_t}_W) = 0$, and so $\mu_t = \lambda_i$, as claimed. 

\medskip

This completes the proof of Theorems \ref{thm:disconnected} and \ref{thm:disconnected_H}.
\end{proof}

\color{black}

\begin{remark}
As mentioned above, the Coxeter group $W$ with $(p,q)=(p_0,q_0)$ or $p=p_0$ is the geometric reflection group of a compact hyperbolic $d$-polytope. Thus, by the same reasoning as in Example 2.3.1 of Januszkiewicz--\'{S}wi{\k{a}}tkowski \cite{janu} (cf. Proposition \ref{prop:not_lattice}), if $(p,q)\neq (p_0,q_0)$ (resp. $p \neq p_0$), then $W_{p,q}$ (resp. $W_p$) is \emph{not} quasi-isometric to the hyperbolic space $\mathbb{H}^d$. 
\end{remark}



\begin{remark} Recently, Davis asked the following question (see Question 2.15 of \cite{MR3445448}): \emph{
Is any abstract geometric reflection group of dimension $d$ isomorphic to a projective Coxeter group in $\mathrm{SL}^{\pm}_{d+1}(\mathbb{R})$?
}
He also conjectured that almost certainly the answer is \emph{no}. 

The counterexamples can be constructed as follows. Let $W$ be a Coxeter group with Coxeter diagram in Tables \ref{Hexamples_dim4}, \ref{examples_dim4}, \ref{examples_dim5}, \ref{examples_dim6}, \ref{examples_dim7}, \ref{examples_dim8}.
By Proposition \ref{prop:nerve_are_sphere}, $W$ is an abstract geometric reflection group of dimension $d$. We now claim that if the underlying graph of the diagram $\mathsf{G}_W$ is a \emph{tree}, then $W$ is not isomorphic to a projective Coxeter group in $\mathrm{SL}^{\pm}_{d+1}(\mathbb{R})$. Indeed, suppose by contradiction that $W$ is isomorphic to a projective Coxeter group $\Gamma_P$ generated by reflections $(\sigma_s = \mathrm{Id} - \alpha_s \otimes b_s)_{s \in S}$ of $\mathrm{SL}^{\pm}_{d+1}(\mathbb{R})$. Then the Cartan matrix $A = ( \alpha_s (b_t) )_{s,t \in S}$ of $P$ is of rank $\leqslant d+1$. Since $\mathsf{G}_W$ has no edge of label $\infty$ and its underlying graph is a tree, by Proposition 20 of Vinberg \cite{bible}, we may assume that the Cartan matrix $A$ is symmetric, i.e. $A = \mathsf{C}_W$. However, it is impossible since by Propositions \ref{prop:signature} and \ref{prop:signatureH}, the matrix $\mathsf{C}_W$ is of rank $d+2$. We finally note that Tables \ref{Hexamples_dim4}, \ref{examples_dim4}, \ref{examples_dim6}, \ref{examples_dim8} do contain Coxeter diagrams whose underlying graphs are trees.

\end{remark}

\clearpage

\newcommand{\echelle}{0.45}

\appendix

\section{The tables of the Coxeter groups}

The labels $\mathcal{E}_i$, $\mathcal{Q}_j$, $\mathcal{T}_k$ provide the information of relationships between Coxeter diagrams in Tables \ref{examples_Esselmann}, \ref{Hexamples_dim4}, \ref{examples_dim4}, in Tables \ref{table:Tumarkin_dim4}, \ref{table:barbot2_not-poincare_hyp_dim4}, \ref{table:barbot2_Quasi_Fuchsian_dim4} and in Tables \ref{table:Tumarkin_dim6}, \ref{table:barbot2_not-poincare_hyp_dim6}, \ref{table:barbot2_Quasi_Fuchsian_dim6}, \ie the diagrams in the item with same label are the same except for different ranges of $p$ and $q$.

\subsection{Esselmann’s examples}
Each Coxeter group in Table \ref{examples_Esselmann} admits a Fuchsian representation.


\begin{table}[ht!]
\begin{center}
\begin{tabular}{rrrrrrr}
\begin{tikzpicture}[thick,scale=\echelle, every node/.style={transform shape}]
\node[draw, inner sep=5pt] (1) at (-2.6,1.1) {{\Large $\mathcal{E}_1$}};
\node[draw,circle] (3) at (0,0) {};
\node[draw,circle] (2) at (-1.5,0.866) {};
\node[draw,circle] (1) at (-1.5,-0.866) {};
\node[draw,circle,fill=black] (4) at (1.732,0) {};
\node[draw,circle,fill=black] (5) at (1.732+1.5,0) {};
\node[draw,circle,fill=black] (6) at (1.732+3,0) {};

\draw (2) -- (3)  node[above,midway] {$4$};
\draw (3)--(1) node[below,midway] {$5$};
\draw (4) -- (5) node[above,midway] {};
\draw (5) -- (6) node[above,midway] {$10$};
\draw (3) -- (4) node[above,midway] {};
\draw (1.0,-1.5) node[]{} ;
\end{tikzpicture}
&
&

\begin{tikzpicture}[thick,scale=\echelle, every node/.style={transform shape}]
\node[draw, inner sep=5pt] (1) at (-2.6,1.1) {{\Large $\mathcal{E}_2$}};
\node[draw,circle] (3) at (0,0) {};
\node[draw,circle] (2) at (-1.5,0.866) {};
\node[draw,circle] (1) at (-1.5,-0.866) {};

\node[draw,circle,fill=black] (4) at (1.732,0) {};
\node[draw,circle,fill=black] (5) at (1.732+1.5,0.866) {};
\node[draw,circle,fill=black] (6) at (1.732+1.5,-0.866) {};

\draw (2) -- (3)  node[above,midway] {$4$};
\draw (3)--(1) node[below,midway] {$5$};
\draw (4) -- (5) node[above,near start] {};
\draw (5) -- (6) node[right,midway] {$5$};
\draw (3) -- (4) node[above,midway] {};
\draw (4) -- (6) node[below,near start] {};
\draw (1.0,-1.5) node[]{} ;
\end{tikzpicture}

&&

\begin{tikzpicture}[thick,scale=\echelle, every node/.style={transform shape}]
\node[draw, inner sep=5pt] (1) at (-2.6,1.1) {{\Large $\mathcal{E}_3$}};
\node[draw,circle] (3) at (0,0) {};
\node[draw,circle] (2) at (-1.5,0.866) {};
\node[draw,circle] (1) at (-1.5,-0.866) {};
\node[draw,circle,fill=black] (4) at (1.732,0) {};
\node[draw,circle,fill=black] (5) at (1.732+1.5,0) {};
\node[draw,circle,fill=black] (6) at (1.732+3,0) {};

\draw (1) -- (2)  node[midway,left] {$4$};
\draw (2) -- (3)  node[above,midway] {};
\draw (3)--(1) node[below,midway] {};
\draw (4) -- (5) node[above,midway] {};
\draw (5) -- (6) node[above,midway] {$8$};
\draw (3) -- (4) node[above,midway] {$4$};
\draw (1.0,-1.5) node[]{} ;
\end{tikzpicture}

&&

\begin{tikzpicture}[thick,scale=\echelle, every node/.style={transform shape}]
\node[draw, inner sep=5pt] (1) at (-2.6,1.1) {{\Large $\mathcal{E}_4$}};
\node[draw,circle] (3) at (0,0) {};
\node[draw,circle] (2) at (-1.5,0.866) {};
\node[draw,circle] (1) at (-1.5,-0.866) {};

\node[draw,circle,fill=black] (4) at (1.732,0) {};
\node[draw,circle,fill=black] (5) at (1.732+1.5,0.866) {};
\node[draw,circle,fill=black] (6) at (1.732+1.5,-0.866) {};

\draw (1) -- (2)  node[midway,left] {$4$};
\draw (2) -- (3)  node[above,midway] {};
\draw (3)--(1) node[below,midway] {};
\draw (4) -- (5) node[above,near start] {};
\draw (5) -- (6) node[right,midway] {$4$};
\draw (3) -- (4) node[above,midway] {$4$};
\draw (4) -- (6) node[below,near start] {};
\draw (1.0,-1.5) node[]{} ;
\end{tikzpicture}

\\
\end{tabular}
\caption{Geometric reflection groups of dimension $d=4$ and of rank $d+2$}
\label{examples_Esselmann}
\end{center}
\end{table}


\subsection{Examples to Theorem \ref{thm:tits_rep_H}}

Each Coxeter group in Table \ref{Hexamples_dim4} admits a quasi-Fuchsian representation into $\mathrm{PO}_{5,1}(\mathbb{R})$.


\begin{table}[ht!]
\begin{center}
\begin{tabular}{ccccccc}
\begin{tikzpicture}[thick,scale=\echelle, every node/.style={transform shape}]
\node[draw, inner sep=5pt] (1) at (-2.6,1.1) {{\Large $\mathcal{E}_1$}};
\node[draw,circle] (3) at (0,0) {};
\node[draw,circle] (2) at (-1.5,0.866) {};
\node[draw,circle] (1) at (-1.5,-0.866) {};
\node[draw,circle,fill=black] (4) at (1.732,0) {};
\node[draw,circle,fill=black] (5) at (1.732+1.5,0) {};
\node[draw,circle,fill=black] (6) at (1.732+3,0) {};

\draw (2) -- (3)  node[above,midway] {$4$};
\draw (3)--(1) node[below,midway] {$5$};
\draw (4) -- (5) node[above,midway] {};
\draw (5) -- (6) node[above,midway] {$p$};
\draw (3) -- (4) node[above,midway] {};
\draw (1.0,-1.5) node[]{$p = 7,8,9$} ;
\end{tikzpicture}
&&
\begin{tikzpicture}[thick,scale=\echelle, every node/.style={transform shape}]
\node[draw, inner sep=5pt] (1) at (-2.6,1.1) {{\Large $\mathcal{Q}_1$}};
\node[draw,circle] (3) at (0,0) {};
\node[draw,circle] (2) at (-1.5,0.866) {};
\node[draw,circle] (1) at (-1.5,-0.866) {};
\node[draw,circle,fill=black] (4) at (1.732,0) {};
\node[draw,circle,fill=black] (5) at (1.732+1.5,0) {};
\node[draw,circle,fill=black] (6) at (1.732+3,0) {};

\draw (2) -- (3)  node[above,midway] {$5$};
\draw (3)--(1) node[below,midway] {$5$};
\draw (4) -- (5) node[above,midway] {};
\draw (5) -- (6) node[above,midway] {$p$};
\draw (3) -- (4) node[above,midway] {};
\draw (1.0,-1.5) node[]{$p = 7,8$} ;
\end{tikzpicture}
&&
\begin{tikzpicture}[thick,scale=\echelle, every node/.style={transform shape}]
\node[draw, inner sep=5pt] (1) at (-2.6,1.1) {{\Large $\mathcal{E}_2$}};
\node[draw,circle] (3) at (0,0) {};
\node[draw,circle] (2) at (-1.5,0.866) {};
\node[draw,circle] (1) at (-1.5,-0.866) {};

\node[draw,circle,fill=black] (4) at (1.732,0) {};
\node[draw,circle,fill=black] (5) at (1.732+1.5,0.866) {};
\node[draw,circle,fill=black] (6) at (1.732+1.5,-0.866) {};

\draw (2) -- (3)  node[above,midway] {$4$};
\draw (3)--(1) node[below,midway] {$5$};
\draw (4) -- (5) node[above,near start] {};
\draw (5) -- (6) node[right,midway] {$4$};
\draw (3) -- (4) node[above,midway] {};
\draw (4) -- (6) node[below,near start] {};
\draw (1.0,-1.5) node[]{\phantom{$p$}} ;
\end{tikzpicture}
&&
\begin{tikzpicture}[thick,scale=\echelle, every node/.style={transform shape}]
\node[draw, inner sep=5pt] (1) at (-2.6,1.1) {{\Large $\mathcal{Q}_2$}};
\node[draw,circle] (3) at (0,0) {};
\node[draw,circle] (2) at (-1.5,0.866) {};
\node[draw,circle] (1) at (-1.5,-0.866) {};

\node[draw,circle,fill=black] (4) at (1.732,0) {};
\node[draw,circle,fill=black] (5) at (1.732+1.5,0.866) {};
\node[draw,circle,fill=black] (6) at (1.732+1.5,-0.866) {};

\draw (2) -- (3)  node[above,midway] {$5$};
\draw (3)--(1) node[below,midway] {$5$};
\draw (4) -- (5) node[above,near start] {};
\draw (5) -- (6) node[right,midway] {$4$};
\draw (3) -- (4) node[above,midway] {};
\draw (4) -- (6) node[below,near start] {};
\draw (1.0,-1.5) node[]{\phantom{$p$}} ;
\end{tikzpicture}
\\
\end{tabular}
\caption{Abstract geometric reflection groups of dimension $d= 4$ and of rank $d+2$}
\label{Hexamples_dim4}
\end{center}
\end{table}

\clearpage

\subsection{Examples to Theorems \ref{thm:specifications} and \ref{thm:tits_rep}}

Each Coxeter group in Tables \ref{examples_dim4}, \ref{examples_dim5}, \ref{examples_dim6}, \ref{examples_dim7} and \ref{examples_dim8} admits a strictly GHC-regular representation.


\begin{table}[ht!]
\begin{center}
\begin{tabular}{ccccccc}
\begin{tikzpicture}[thick,scale=\echelle, every node/.style={transform shape}]
\node[draw, inner sep=5pt] (1) at (-2.6,1.1) {{\Large $\mathcal{E}_1$}};
\node[draw,circle] (3) at (0,0) {};
\node[draw,circle] (2) at (-1.5,0.866) {};
\node[draw,circle] (1) at (-1.5,-0.866) {};
\node[draw,circle,fill=black] (4) at (1.732,0) {};
\node[draw,circle,fill=black] (5) at (1.732+1.5,0) {};
\node[draw,circle,fill=black] (6) at (1.732+3,0) {};

\draw (2) -- (3)  node[above,midway] {$4$};
\draw (3)--(1) node[below,midway] {$5$};
\draw (4) -- (5) node[above,midway] {};
\draw (5) -- (6) node[above,midway] {$p \geqslant 11$};
\draw (3) -- (4) node[above,midway] {};
\end{tikzpicture}
&&
\begin{tikzpicture}[thick,scale=\echelle, every node/.style={transform shape}]
\node[draw, inner sep=5pt] (1) at (-2.6,1.1) {{\Large $\mathcal{Q}_1$}};
\node[draw,circle] (3) at (0,0) {};
\node[draw,circle] (2) at (-1.5,0.866) {};
\node[draw,circle] (1) at (-1.5,-0.866) {};
\node[draw,circle,fill=black] (4) at (1.732,0) {};
\node[draw,circle,fill=black] (5) at (1.732+1.5,0) {};
\node[draw,circle,fill=black] (6) at (1.732+3,0) {};

\draw (2) -- (3)  node[above,midway] {$5$};
\draw (3)--(1) node[below,midway] {$5$};
\draw (4) -- (5) node[above,midway] {};
\draw (5) -- (6) node[above,midway] {$p \geqslant 9$};
\draw (3) -- (4) node[above,midway] {};
\end{tikzpicture}
&&
\begin{tikzpicture}[thick,scale=\echelle, every node/.style={transform shape}]
\node[draw, inner sep=5pt] (1) at (-2.6,1.1) {{\Large $\mathcal{E}_2$}};
\node[draw,circle] (3) at (0,0) {};
\node[draw,circle] (2) at (-1.5,0.866) {};
\node[draw,circle] (1) at (-1.5,-0.866) {};

\node[draw,circle,fill=black] (4) at (1.732,0) {};
\node[draw,circle,fill=black] (5) at (1.732+1.5,0.866) {};
\node[draw,circle,fill=black] (6) at (1.732+1.5,-0.866) {};

\draw (2) -- (3)  node[above,midway] {$4$};
\draw (3)--(1) node[below,midway] {$5$};
\draw (4) -- (5) node[above,near start] {};
\draw (5) -- (6) node[right,midway] {$p \geqslant 6$};
\draw (3) -- (4) node[above,midway] {};
\draw (4) -- (6) node[below,near start] {};
\end{tikzpicture}
&&
\begin{tikzpicture}[thick,scale=\echelle, every node/.style={transform shape}]
\node[draw, inner sep=5pt] (1) at (-2.6,1.1) {{\Large $\mathcal{Q}_2$}};
\node[draw,circle] (3) at (0,0) {};
\node[draw,circle] (2) at (-1.5,0.866) {};
\node[draw,circle] (1) at (-1.5,-0.866) {};

\node[draw,circle,fill=black] (4) at (1.732,0) {};
\node[draw,circle,fill=black] (5) at (1.732+1.5,0.866) {};
\node[draw,circle,fill=black] (6) at (1.732+1.5,-0.866) {};

\draw (2) -- (3)  node[above,midway] {$5$};
\draw (3)--(1) node[below,midway] {$5$};
\draw (4) -- (5) node[above,near start] {};
\draw (5) -- (6) node[right,midway] {$p \geqslant 5$};
\draw (3) -- (4) node[above,midway] {};
\draw (4) -- (6) node[below,near start] {};
\end{tikzpicture}
\\
\end{tabular}
\caption{Abstract geometric reflection groups of dimension $d=4$ and of rank $d+2$}
\label{examples_dim4}
\end{center}
\end{table}

\clearpage

\newcommand{\echellee}{0.57}

\begin{table}[ht!]
\begin{center}
\begin{tabular}{ccc}
\begin{tikzpicture}[thick,scale=\echellee, every node/.style={transform shape}]
\node[draw,circle] (1) at (0,0) {};
\node[draw,circle] (2) at (1.06,1.06) {};
\node[draw,circle] (3) at (1.06,-1.06) {};
\node[draw,circle] (4) at (2.12,0) {};

\node[draw,circle,fill=black] (5) at (3.62,0) {};
\node[draw,circle,fill=black] (6) at (5.12,0) {};
\node[draw,circle,fill=black] (7) at (6.62,0) {};

\draw (1) -- (3) node[below,near start] {};
\draw (2) -- (1) node[above,near end] {$4$};
\draw (2) -- (4);
\draw (3)--(4);
\draw (4) -- (5) node[above,near start] {};
\draw (5) -- (6) node[right,midway] {};
\draw (6) -- (7) node[above,midway] {$p \geqslant 7$};
\end{tikzpicture}
&&
\begin{tikzpicture}[thick,scale=\echellee, every node/.style={transform shape}]
\node[draw,circle] (1) at (0,0) {};
\node[draw,circle] (2) at (1.06,1.06) {};
\node[draw,circle] (3) at (1.06,-1.06) {};
\node[draw,circle] (4) at (2.12,0) {};

\node[draw,circle,fill=black] (5) at (3.62,0) {};
\node[draw,circle,fill=black] (6) at (3.62+1.5,0.866) {};
\node[draw,circle,fill=black] (7) at (3.62+1.5,-0.866) {};

\draw (1) -- (3) node[below,near start] {};
\draw (2) -- (1) node[above,near end] {$4$};
\draw (2) -- (4);
\draw (3)--(4);
\draw (4) -- (5) node[above,near start] {};
\draw (5) -- (6) node[right,midway] {};
\draw (6) -- (7) node[right,midway] {$p \geqslant 4$};
\draw (7) -- (5);
\end{tikzpicture}
\end{tabular}
\end{center}
\caption{Abstract geometric reflection groups of dimension $d=5$ and of rank $d+2$}
\label{examples_dim5}
\begin{center}
\begin{tabular}{ccccc}
\begin{tikzpicture}[thick,scale=\echellee, every node/.style={transform shape}]
\node[draw,circle] (3) at (1,0) {};
\node[draw,circle] (1) at (0.309,-0.951){};
\node[draw,circle] (2) at (0.309,0.951){};
\node[draw,circle] (-1) at (-0.809,-0.587){};
\node[draw,circle] (-2) at (-0.809,0.587){};

\node[draw,circle,fill=black] (4) at (2.5,0) {};
\node[draw,circle,fill=black] (5) at (2.5+1.5,0.866) {};
\node[draw,circle,fill=black] (6) at (2.5+1.5,-0.866) {};

\draw (3)--(1);
\draw (1) -- (-1);
\draw (-1) -- (-2) node[left,midway] {$4$};
\draw (-2) -- (2);
\draw (2) -- (3);
\draw (4) -- (5);
\draw (4) -- (6);
\draw (5) -- (6) node[right,midway] {$p \geqslant 4$};
\draw (3) -- (4) node[above,midway] {};
\end{tikzpicture}
&&
\begin{tikzpicture}[thick,scale=\echellee, every node/.style={transform shape}]
\node[draw,circle] (1) at (-5,0){};
\node[draw,circle] (2) at (-3.5,0){};
\node[draw,circle] (3) at (-2,0){};
\node[draw,circle] (4) at (-0.5,0){};
\node[draw,circle] (5) at (1,0) {};

\node[draw,circle,fill=black] (6) at (2.5,0) {};
\node[draw,circle,fill=black] (7) at (2.5+1.5,0.866) {};
\node[draw,circle,fill=black] (8) at (2.5+1.5,-0.866) {};

\draw (1)--(2) node[above,midway] {$5$};
\draw (2) -- (3);
\draw (3) -- (4);
\draw (4) -- (5) ;

\draw (6) -- (5);

\draw (8) -- (6);
\draw (7) -- (6) ;
\draw (7) -- (8) node[right,midway] {$p \geqslant 4$};
\end{tikzpicture}
\\
\\
\end{tabular}
\end{center}
\caption{Abstract geometric reflection groups of dimension $d=6$ and of rank $d+2$}
\label{examples_dim6}
\begin{center}
%
\end{center}
\caption{An abstract geometric reflection group of dimension $d=7$ and of rank $d+2$}
\label{examples_dim7}
\begin{center}
\begin{tabular}{ccc}
\begin{tikzpicture}[thick,scale=\echellee, every node/.style={transform shape}]
\node[draw,circle] (3) at (1,0) {};
\node[draw,circle] (1) at (0.309,-0.951){};
\node[draw,circle] (2) at (0.309,0.951){};
\node[draw,circle] (-1) at (-0.809,-0.587){};
\node[draw,circle] (-2) at (-0.809,0.587){};

\node[draw,circle,fill=black] (4) at (2.5,0) {};
\node[draw,circle,fill=black] (5) at (4,0) {};
\node[draw,circle,fill=black] (6) at (5.5,0) {};
\node[draw,circle,fill=black] (7) at (7,0) {};
\node[draw,circle,fill=black] (8) at (8.5,0) {};

\draw (3)--(1);
\draw (1) -- (-1);
\draw (-1) -- (-2) node[left,midway] {$4$};
\draw (-2) -- (2);
\draw (2) -- (3);
\draw (3) -- (4) node[above,midway] {};

\draw (4) -- (5);
\draw (5) -- (6) node[right,midway] {};
\draw (6) -- (7) node[above,midway] {};
\draw (7) -- (8) node[above,midway] {$5$};
\end{tikzpicture}
&&
\begin{tikzpicture}[thick,scale=\echellee, every node/.style={transform shape}]
\node[draw,circle] (1) at (-5,0){};
\node[draw,circle] (2) at (-3.5,0){};
\node[draw,circle] (3) at (-2,0){};
\node[draw,circle] (4) at (-0.5,0){};
\node[draw,circle] (5) at (1,0) {};

\node[draw,circle,fill=black] (4a) at (2.5,0) {};
\node[draw,circle,fill=black] (5a) at (4,0) {};
\node[draw,circle,fill=black] (6a) at (5.5,0) {};
\node[draw,circle,fill=black] (7a) at (7,0) {};
\node[draw,circle,fill=black] (8a) at (8.5,0) {};

\draw (1)--(2) node[above,midway] {$5$};
\draw (2) -- (3);
\draw (3) -- (4);
\draw (4) -- (5);
\draw (4a) -- (5);
\draw (4a) -- (5a);
\draw (5a) -- (6a) node[right,midway] {};
\draw (6a) -- (7a) ;
\draw (7a) -- (8a) node[above,midway] {$5$};
\draw (0,-1) node[]{$\,$} ;

\end{tikzpicture}
\end{tabular}
\end{center}
\caption{Abstract geometric reflection groups of dimension $d=8$ and of rank $d+2$}
\label{examples_dim8}
\end{table}


\subsection{Tumarkin’s Examples}

Each Coxeter group in Tables \ref{table:Tumarkin_dim4} and \ref{table:Tumarkin_dim6} admits a Fuchsian representation.

\newcommand{\scalee}{0.65}

\begin{table}[ht!]
\begin{tabular}{cc cc ccc}
\begin{tikzpicture}[thick,scale=\scalee, every node/.style={transform shape}]
\node[draw, inner sep=5pt] (1) at (-2.7,1.1) {{$\mathcal{T}_1$}};
\node[draw,circle] (1) at (36-18:1.0514) {};
\node[draw,circle] (2) at (108-18:1.0514){};
\node[draw,circle] (3) at (180-18:1.0514){};
\node[draw,circle,fill=\Dshaded] (4) at (252-18:1.0514){};
\node[draw,circle,fill=\Lshaded] (5) at (324-18:1.0514){};
\draw (1)--(2)--(3)--(4)--(5)--(1);
\node[draw,circle,right=1cm of 1] (6) {};
\node[draw,circle,left=1cm of 3] (7) {};
\draw (1)--(6) node[above,midway] {$10$};
\draw (3)--(7) node[above,midway] {$10$};
\draw (4)--(5) node[below,midway] {$\infty$};
\end{tikzpicture}
&&
\begin{tikzpicture}[thick,scale=\scalee, every node/.style={transform shape}]
\node[draw, inner sep=5pt] (1) at (-1.6,1.1) {{$\mathcal{T}_2$}};
\node[draw,circle] (1) at (36-18:1.0514) {};
\node[draw,circle] (2) at (108-18:1.0514){};
\node[draw,circle] (3) at (180-18:1.0514){};
\node[draw,circle,fill=\Dshaded] (4) at (252-18:1.0514){};
\node[draw,circle,fill=\Lshaded] (5) at (324-18:1.0514){};
\draw (1)--(2)--(3)--(4)--(5)--(1);
\node[draw,circle,right=1cm of 1] (6) {};
\node[draw,circle] (7) at ($(0,0)$) {};
\draw (1)--(6) node[above,midway] {$10$};
\draw (4)--(5) node[below,midway] {$\infty$};
\draw (4)--(3);
\draw (5)--(1);
\draw (7)--(2);
\draw (7)--(3) node[midway] {$5$};
\draw (7)--(4) ;
\end{tikzpicture}
&&
\begin{tikzpicture}[thick,scale=\scalee, every node/.style={transform shape}]
\node[draw, inner sep=5pt] (1) at (-1.1,2.0) {{$\mathcal{T}_3$}};
\node[draw,circle,fill=\Dshaded] (1) at (0,0) {};
\node[draw,circle] (2) at (0,2) {};
\node[draw,circle] (4) at (2,2) {};
\node[draw,circle] (3) at (1,2) {};
\node[draw,circle] (5) at (2,1) {};
\node[draw,circle,fill=\Lshaded] (6) at (2,0) {};
\draw (1)--(2)--(3)--(4)--(5)--(6)--(1);
\node[draw,circle] (7) at (3,1) {};
\draw (5)--(7) node[above,midway] {$8$};
\draw (2)--(3) node[above,midway] {$8$};
\draw (4)--(3) node[above,midway] {$4$};
\draw (1)--(2) node[left,midway] {$4$};
\draw (1)--(6) node[below,midway] {$\infty$};

\end{tikzpicture}
\\
\end{tabular}
\caption{Geometric reflection groups of dimension $d=4$ and of rank $d+3$}
\label{table:Tumarkin_dim4}
%
\caption{Geometric reflection groups of dimension $d=6$ and of rank $d+3$}
\label{table:Tumarkin_dim6}
\end{table}

\color{black}

\clearpage


\subsection{Examples to Theorem \ref{thm:disconnected}}

The space of strictly GHC-regular characters of each Coxeter group in Tables \ref{table:barbot2_not-poincare_hyp_dim4} and \ref{table:barbot2_not-poincare_hyp_dim6} is disconnected.

\begin{table}[ht!]
\begin{tabular}{cc cc ccc}
\begin{tikzpicture}[thick,scale=\scalee, every node/.style={transform shape}]
\node[draw, inner sep=5pt] (1) at (-2.7,1.1) {{$\mathcal{T}_1$}};
\node[draw,circle] (1) at (36-18:1.0514) {};
\node[draw,circle] (2) at (108-18:1.0514){};
\node[draw,circle] (3) at (180-18:1.0514){};
\node[draw,circle,fill=\Dshaded] (4) at (252-18:1.0514){};
\node[draw,circle,fill=\Lshaded] (5) at (324-18:1.0514){};
\draw (1)--(2)--(3)--(4)--(5)--(1);
\node[draw,circle,right=1cm of 1] (6) {};
\node[draw,circle,left=1cm of 3] (7) {};
\draw (1)--(6) node[above,midway] {$q \geqslant 9$};
\draw (3)--(7) node[above,midway] {$p \geqslant 9$};
\draw (4)--(5) node[below,midway] {$\infty$};
\draw (-0.2,-2) node[]{$\{p,q\} \neq \{9,9\}, \{9,10\}, \dotsc, \{9,18\}, \{10,10\} $} ;
\end{tikzpicture}
&&
\begin{tikzpicture}[thick,scale=\scalee, every node/.style={transform shape}]
\node[draw, inner sep=5pt] (1) at (-1.6,1.1) {{$\mathcal{T}_2$}};
\node[draw,circle] (1) at (36-18:1.0514) {};
\node[draw,circle] (2) at (108-18:1.0514){};
\node[draw,circle] (3) at (180-18:1.0514){};
\node[draw,circle,fill=\Dshaded] (4) at (252-18:1.0514){};
\node[draw,circle,fill=\Lshaded] (5) at (324-18:1.0514){};
\draw (1)--(2)--(3)--(4)--(5)--(1);
\node[draw,circle,right=1cm of 1] (6) {};
\node[draw,circle] (7) at ($(0,0)$) {};
\draw (1)--(6) node[above,midway] {$p \geqslant 11$};
\draw (4)--(5) node[below,midway] {$\infty$};
\draw (4)--(3);
\draw (5)--(1);
\draw (7)--(2);
\draw (7)--(3) node[midway] {$5$};
\draw (7)--(4) ;
\draw (0,-2) node[]{\phantom{$p_0 := 10$}} ;
\end{tikzpicture}
&&
\begin{tikzpicture}[thick,scale=\scalee, every node/.style={transform shape}]
\node[draw, inner sep=5pt] (1) at (-1.1,2.0) {{$\mathcal{T}_3$}};
\node[draw,circle,fill=\Dshaded] (1) at (0,0) {};
\node[draw,circle] (2) at (0,2) {};
\node[draw,circle] (4) at (2,2) {};
\node[draw,circle] (3) at (1,2) {};
\node[draw,circle] (5) at (2,1) {};
\node[draw,circle,fill=\Lshaded] (6) at (2,0) {};
\draw (1)--(2)--(3)--(4)--(5)--(6)--(1);
\node[draw,circle] (7) at (3,1) {};
\draw (5)--(7) node[above,midway] {$q \geqslant 8$};
\draw (2)--(3) node[above,midway] {$p \geqslant 7$};
\draw (4)--(3) node[above,midway] {$4$};
\draw (1)--(2) node[left,midway] {$4$};
\draw (1)--(6) node[below,midway] {$\infty$};
\draw (1.3,-1) node[]{$(p,q) \neq (7,8), (7,9), (8,8)$} ;

\end{tikzpicture}
\\
\end{tabular}
\caption{Abstract geometric reflection groups of dimension $d=4$ and of rank $d+3$}
\label{table:barbot2_not-poincare_hyp_dim4}
%
\caption{Abstract geometric reflection groups of dimension $d=6$ and of rank $d+3$}
\label{table:barbot2_not-poincare_hyp_dim6}
\end{table}

\subsection{Examples to Theorem \ref{thm:disconnected_H}}

The space of quasi-Fuchsian characters of each Coxeter group in Tables \ref{table:barbot2_Quasi_Fuchsian_dim4} and \ref{table:barbot2_Quasi_Fuchsian_dim6} is disconnected.

\begin{table}[ht!]
\begin{tabular}{cc cc ccc}
\begin{tikzpicture}[thick,scale=\scalee, every node/.style={transform shape}]
\node[draw, inner sep=5pt] (1) at (-2.7,1.1) {{$\mathcal{T}_1$}};
\node[draw,circle] (1) at (36-18:1.0514) {};
\node[draw,circle] (2) at (108-18:1.0514){};
\node[draw,circle] (3) at (180-18:1.0514){};
\node[draw,circle,fill=\Dshaded] (4) at (252-18:1.0514){};
\node[draw,circle,fill=\Lshaded] (5) at (324-18:1.0514){};
\draw (1)--(2)--(3)--(4)--(5)--(1);
\node[draw,circle,right=1cm of 1] (6) {};
\node[draw,circle,left=1cm of 3] (7) {};
\draw (1)--(6) node[above,midway] {$q$};
\draw (3)--(7) node[above,midway] {$p$};
\draw (4)--(5) node[below,midway] {$\infty$};
\draw (-0.2,-2) node[]{$\{p,q\} = \{7,13\}, \{8,10\}, \{8,11\}, \{9,9\}, \{9,10\}$} ;
\end{tikzpicture}
&&
\begin{tikzpicture}[thick,scale=\scalee, every node/.style={transform shape}]
\node[draw, inner sep=5pt] (1) at (-1.6,1.1) {{$\mathcal{T}_2$}};
\node[draw,circle] (1) at (36-18:1.0514) {};
\node[draw,circle] (2) at (108-18:1.0514){};
\node[draw,circle] (3) at (180-18:1.0514){};
\node[draw,circle,fill=\Dshaded] (4) at (252-18:1.0514){};
\node[draw,circle,fill=\Lshaded] (5) at (324-18:1.0514){};
\draw (1)--(2)--(3)--(4)--(5)--(1);
\node[draw,circle,right=1cm of 1] (6) {};
\node[draw,circle] (7) at ($(0,0)$) {};
\draw (1)--(6) node[above,midway] {$p$};
\draw (4)--(5) node[below,midway] {$\infty$};
\draw (4)--(3);
\draw (5)--(1);
\draw (7)--(2);
\draw (7)--(3) node[midway] {$5$};
\draw (7)--(4) ;
\draw (0,-2) node[]{$p = 8, 9$} ;
\end{tikzpicture}
&&
\begin{tikzpicture}[thick,scale=\scalee, every node/.style={transform shape}]
\node[draw, inner sep=5pt] (1) at (-1.1,2.0) {{$\mathcal{T}_3$}};
\node[draw,circle,fill=\Dshaded] (1) at (0,0) {};
\node[draw,circle] (2) at (0,2) {};
\node[draw,circle] (4) at (2,2) {};
\node[draw,circle] (3) at (1,2) {};
\node[draw,circle] (5) at (2,1) {};
\node[draw,circle,fill=\Lshaded] (6) at (2,0) {};
\draw (1)--(2)--(3)--(4)--(5)--(6)--(1);
\node[draw,circle] (7) at (3,1) {};
\draw (5)--(7) node[above,midway] {$q$};
\draw (2)--(3) node[above,midway] {$p$};
\draw (4)--(3) node[above,midway] {$4$};
\draw (1)--(2) node[left,midway] {$4$};
\draw (1)--(6) node[below,midway] {$\infty$};
\draw (1.3,-1) node[]{$(p,q) = (7,8), (9,7)$} ;

\end{tikzpicture}
\\
\end{tabular}
\caption{Abstract geometric reflection groups of dimension $d=4$ and of rank $d+3$}
\label{table:barbot2_Quasi_Fuchsian_dim4}
%
\caption{Abstract geometric reflection groups of dimension $d=6$ and of rank $d+3$}
\label{table:barbot2_Quasi_Fuchsian_dim6}
\end{table}

\clearpage

\section{The spherical, affine and Lannér diagrams}\label{classi_diagram}

For the reader’s convenience, we reproduce below the list of the irreducible spherical, irreducible affine and Lannér diagrams.

\begin{table}[ht!]
\centering
\begin{minipage}[b]{7.5cm}
\centering
%
\caption{The irreducible spherical Coxeter diagrams}
\label{spheri_diag}
\end{minipage}
\begin{minipage}[t]{7.5cm}
\centering
%
\caption{The irreducible affine Coxeter diagrams}
\label{affi_diag}
\end{minipage}
\end{table}

\begin{table}[ht!]
%
\caption{The Lannér Coxeter diagrams}
\label{table:Lanner}
\end{table}

\newpage

\bibliographystyle{alpha}

\newcommand{\etalchar}[1]{$^{#1}$}

\end{document}